\documentclass[a4paper,openbib,12pt]{article}
\usepackage{geometry}
\usepackage{amsmath}
\usepackage{amssymb}
\usepackage{amsfonts}

\newtheorem{theorem}{Theorem}[section]

\newtheorem{corollary}[theorem]{Corollary}

\newtheorem{definition}[theorem]{Definition}

\newtheorem{lemma}[theorem]{Lemma}

\newtheorem{proposition}[theorem]{Proposition}
\newtheorem{remark}[theorem]{Remark}

\newenvironment{proof}[1][Proof]{\noindent\textbf{#1.} }{\ \rule{0.5em}{0.5em}}
\input{tcilatex}
\begin{document}

\title{Length functions on groups and rigidity}
\author{Shengkui Ye}
\maketitle

\begin{abstract}
Let $G$ be a group. A function $l:G\rightarrow \lbrack 0,\infty )$ is called
a length function if

(1) $l(g^{n})=|n|l(g)$ for any $g\in G$ and $n\in \mathbb{Z};$

(2) $l(hgh^{-1})=l(g)$ for any $h,g\in G;$ and

(3) $l(ab)\leq l(a)+l(b)$ for \emph{commuting} elements $a,b.$

Such length functions exist in many branches of mathematics, mainly as
stable word lengths, stable norms, smooth measure-theoretic entropy,
translation lengths on $\mathrm{CAT}(0)$ spaces and Gromov $\delta $%
-hyperbolic spaces, stable norms of quasi-cocycles, rotation numbers of
circle homeomorphisms, dynamical degrees of birational maps, absolute values
of Margulis invariants and so on. We study length functions on Lie groups,
Gromov hyperbolic groups, arithmetic subgroups, matrix groups over rings and
Cremona groups. As applications, we prove that every group homomorphism from
an arithmetic subgroup of a simple algebraic $\mathbb{Q}$-group of $\mathbb{Q%
}$-rank at least $2,$ or a finite-index subgroup of the elementary group $%
E_{n}(R)$ $(n\geq 3)$ over an associative ring, or the Cremona group $%
\mathrm{Bir}(P_{\mathbb{C}}^{2})$ to any group $G$ having a purely positive
length function must have its image finite. Here $G$ can be outer
automorphism group $\mathrm{Out}(F_{n})$ of free groups, mapping class group 
$\mathrm{MCG}(\Sigma _{g})$, $\mathrm{CAT}(0)$ groups or Gromov hyperbolic
groups, or the group $\mathrm{Diff}(\Sigma ,\omega )$ of diffeomorphisms of
a hyperbolic closed surface preserving an area form $\omega .$
\end{abstract}

\subsection{Introduction}

The rigidity phenomena have been studied for many years. The famous Margulis
superrigidity implies any group homomorphism between irreducible lattices in
semisimple Lie groups of real rank $\mathrm{rk}_{\mathbb{R}}(G)\geq 2$ are
virtually induced by group homomorphisms between the Lie groups. Therefore,
group homomorphisms from `higher'-rank irreducible lattices to `lower'-rank
irreducible lattices normally have finite images. Farb, Kaimanovich and
Masur \cite{fm} \cite{km} prove that every homomorphism from an
(irreducible) higher rank lattice into the mapping class group $\mathrm{MCG}%
(\Sigma _{g})$ has a finite image. Bridson and Wade \cite{bw} showed that
the same superrigidity remains true if the target is replaced with the outer
automorphism group $\mathrm{Out}(F_{n})$ of the free group$.$ Mimura \cite%
{mi} proves that every homomorphism from Chevalley group over commutative
rings to $\mathrm{MCG}(\Sigma _{g})$ or $\mathrm{Out}(F_{n})$ has a finite
image. Many other rigidity results can be found, e.g. \cite{monod} \cite{bm1}
\cite{bm2} \cite{ha} \cite{py} and \cite{po}. In this article, we study
rigidity phenomena with the notion of length functions.

Let $G$ be a group. We call a function $l:G\rightarrow \lbrack 0,\infty )$ a
length function if

1) $l(g^{n})=|n|l(g)$ for any $g\in G$ and $n\in \mathbb{Z};$

2) $l(hgh^{-1})=l(g)$ for any $h,g\in G;$

3) $l(ab)\leq l(a)+l(b)$ for elements $a,b$ satisfying $ab=ba.$

Such length functions exist in geometric group theory, dynamical systems,
algebra, algebraic geometry and many other branches of mathematics. For
example, the following functions $l$ are length functions (see Section \ref%
{section3} for more examples with details).

\begin{itemize}
\item (The stable word lengths) Let $G$ be a group generated by a symmetric
(not necessarily finite) set $S.$ For any $g\in G,$ the word length $\phi
_{S}(w)=\min \{n\mid g=s_{1}s_{2}\cdots s_{n},$ each $s_{i}\in S\}$ is the
minimal number of elements of $S$ whose product is $g.$ The stable length is
defined as 
\begin{equation*}
l(g)=\lim_{n\rightarrow \infty }\frac{\phi _{S}(g^{n})}{n}.
\end{equation*}

\item (Stable norms) Let $M$ be a compact Riemannian manifold and $G=\mathrm{%
Diff}(M)$ the diffeomorphism group consisting of all self-diffeomorphisms.
For any diffeomorphism $f:M\rightarrow M,$ let 
\begin{equation*}
\Vert f\Vert =\sup_{x\in M}\Vert D_{x}f\Vert ,
\end{equation*}%
where $D_{x}f$ is the induced linear map between tangent spaces $%
T_{x}M\rightarrow T_{f(x)}M.$ Define 
\begin{equation*}
l(f)=\max \{\lim_{n\rightarrow +\infty }\frac{\log \Vert f^{n}\Vert }{n}%
,\lim_{n\rightarrow +\infty }\frac{\log \Vert f^{-n}\Vert }{n}\}.
\end{equation*}

\item (Smooth measure-theoretic entropy) Let $M$ be a $C^{\infty }$ closed
Riemannian manifold and $G=\mathrm{Diff}_{\mu }^{2}(M)$ consisting of
diffeomorphisms of $M$ preserving a Borel probability measure $\mu .$ Let $%
l(f)=h_{\mu }(f)$ be the measure-theoretic entropy, for any $f\in G=\mathrm{%
Diff}_{\mu }^{2}(M)$.

\item (Translation lengths) Let $(X,d)$ be a metric space and $G=\mathrm{Isom%
}(X)$ consisting of isometries $\gamma :X\rightarrow X$. Fix $x\in X,$ define%
\begin{equation*}
l(\gamma )=\lim_{n\rightarrow \infty }\frac{d(x,\gamma ^{n}x)}{n}.
\end{equation*}%
This contains the translation lengths on $\mathrm{CAT}(0)$ spaces and Gromov 
$\delta $-hyperbolic spaces as special cases.

\item (Average norm for quasi-cocycles) Let $G$ be a group and $E$ be a
Hilbert space with an $G$-action by linear isometries. A function $%
f:G\rightarrow E$ is a quasi-{}cocycle if there exists $C>0$ such that 
\begin{equation*}
\Vert f(gh)-f(g)-gf(h)\Vert <C
\end{equation*}%
for any $g,h\in G.$ Let $l:G\rightarrow \lbrack 0,+\infty )$ be defined by 
\begin{equation*}
l(g)=\lim_{n\rightarrow \infty }\frac{\Vert f(g^{n})\Vert }{n}.
\end{equation*}

\item (Rotation numbers of circle homeomorphisms) Let $\mathbb{R}$ be the
real line and $G=\mathrm{Homeo}_{\mathbb{Z}}(\mathbb{R})=\{f\mid f:\mathbb{R}%
\rightarrow \mathbb{R}$ is a monotonically increasing homeomorphism such
that $f(x+n)=f(x)+n$ for any $n\in \mathbb{Z}\}.$ For any $f\in \mathrm{Homeo%
}_{\mathbb{Z}}(\mathbb{R})$ and $x\in \lbrack 0,1),$ the translation number
is defined as 
\begin{equation*}
l(f)=\lim_{n\rightarrow \infty }\frac{f^{n}(x)-x}{n}.
\end{equation*}

\item (Asymptotic distortions) Let $f$ be a $C^{1+bv}$ diffeomorphism of the
closed interval $[0,1]$ or the circle $S^{1}.$ (\textquotedblleft
bv\textquotedblright\ means derivative with finite total variation.) The
asymptotic distortion of $f$ is defined (by Navas \cite{Na}) as 
\begin{equation*}
l(f)=\lim_{n\rightarrow \infty }\frac{1}{n}\mathrm{var}(\log Df^{n}).
\end{equation*}%
This gives a length function $l$ on the group $\mathrm{Diff}^{1+bv}(M)$ of $%
C^{1+bv}$ diffeomorphisms for $M=[0,1]$ or $S^{1}.$

\item (Dynamical degree) Let $\mathbb{C}P^{n}$ be the complex projective
space and $f:\mathbb{C}P^{n}\dashrightarrow \mathbb{C}P^{n}$ be a birational
map given by%
\begin{equation*}
(x_{0}:x_{1}:\cdots :x_{n})\dashrightarrow (f_{0}:f_{1}:\cdots :f_{n}),
\end{equation*}%
where the $f_{i}$'s are homogeneous polynomials of the same degree without
common factors. The degree of $f$ is $\deg f=\deg f_{i}.$ Define%
\begin{equation*}
l(f)=\max \{\lim_{n\rightarrow \infty }\log \deg (f^{n})^{\frac{1}{n}%
},\lim_{n\rightarrow \infty }\log \deg (f^{-n})^{\frac{1}{n}}\}.
\end{equation*}%
This gives a length function $l:\mathrm{Bir}(\mathbb{C}P^{n})\rightarrow
\lbrack 0,+\infty ).$ Here $\mathrm{Bir}(\mathbb{C}P^{n})$ is the group of
birational maps, also called Cremona group.

\item (Margulis invariant) Let $\mathrm{Aff}(\mathbb{R}^{3})=\mathbb{R}%
^{3}\rtimes \mathrm{GL}_{3}(\mathbb{R})$ be the affine automorphism group of
the $3$-dimensional space $\mathbb{R}^{3}$ and 
\begin{eqnarray*}
\phi &:&\Gamma \rightarrow \mathrm{Aff}(\mathbb{R}^{3}), \\
\gamma &\mapsto &(u(\gamma ),\Phi (\gamma )),
\end{eqnarray*}%
a group homomorphism. Let $B$ be the bilinear form on $\mathbb{R}^{3}$
defined by%
\begin{equation*}
B(v,w)=v_{1}w_{1}+v_{2}w_{2}-v_{3}w_{3},
\end{equation*}%
with the isometric group $\mathrm{O}(2,1).$ Suppose that the linear part $%
\Phi (\Gamma )$ is contained in $\mathrm{SO}(2,1)$ and each non-trivial
element $\Phi (\gamma )$ has eigenvalues $\lambda <1<1/\lambda $ for some $%
\lambda >0$ (i.e. $\Phi (\gamma )$ is hyperbolic). Denote by $x^{-}(\gamma
),x^{0}(\gamma ),x^{+}(\gamma )$ the corresponding unit eigenvectors such
that $B(x^{0}(\gamma ),x^{0}(\gamma ))=1$ and $(x^{-}(\gamma ),x^{0}(\gamma
),x^{+}(\gamma ))$ is positively oriented basis. The Margulis invariant
(defined in \cite{M1}, \cite{M2}) is%
\begin{eqnarray*}
\alpha _{\phi } &:&\Gamma \rightarrow \mathbb{R}, \\
\gamma &\mapsto &B(x^{0}(\gamma ),u(\gamma )).
\end{eqnarray*}%
The absolute value function $|\alpha _{\phi }|$ is a length function on $%
\Gamma .$
\end{itemize}

The terminologies of length functions are used a lot in the literature (eg. 
\cite{tao}, \cite{cf}, \cite{conner}). However, they usually mean different
things from ours (in particular, it seems that the subadditivity condition 3
has not been required only for commuting elements before).

Our first observation is the following result on the vanishing of length
functions.

\begin{theorem}
\label{theorem0}Let $G_{A}=\mathbb{Z}^{2}\rtimes _{A}\mathbb{Z}$ be an
abelian-by-cyclic group, where $A\in \mathrm{SL}_{2}(\mathbb{Z})$.

\begin{enumerate}
\item[(i)] When the absolute value of the trace $|\mathrm{tr}(A)|>2,$ any
length function $l:\mathbb{Z}^{2}\rtimes _{A}\mathbb{Z}\rightarrow \mathbb{R}%
_{\geq 0}$ vanishes on $\mathbb{Z}^{2}.$

\item[(ii)] When $|\mathrm{tr}(A)|=2$ and $A\neq I_{2},$ any length function 
$l:\mathbb{Z}^{2}\rtimes _{A}\mathbb{Z}\rightarrow \mathbb{R}_{\geq 0}$
vanishes on the direct summand of $\mathbb{Z}^{2}$ spanned by eigenvectors
of $A$.
\end{enumerate}
\end{theorem}

\begin{corollary}
\label{corllary0}Suppose that the semi-direct product $G_{A}=\mathbb{Z}%
^{2}\rtimes _{A}\mathbb{Z}$ acts on a compact manifold by Lipschitz
homeomorphisms (or $C^{2}$-diffeomorphisms, resp.). The topological entropy $%
h_{top}(g)=0$ (or Lyapunov exponents of $g$ are zero, resp.) for any $g\in 
\mathbb{Z}^{2}$ when $|\mathrm{tr}(A)|>2$ or any eigenvector $g\in \mathbb{Z}%
^{2}$ when $|\mathrm{tr}(A)|=2.$
\end{corollary}

It is well-known that the central element in the integral Heisenberg group $%
G_{A}$ (for $A=%
\begin{bmatrix}
1 & 1 \\ 
0 & 1%
\end{bmatrix}%
$) is distorted in the word metric. Lubotzky, Mozes, Ragunathan \cite{lmr1}
obtain a general distortion result in the word metric for the semi-direct
product $\mathbb{Z}^{2}\rtimes _{A}\mathbb{Z}$. When the Heisenberg group $%
G_{A}$ acts on a $C^{\infty }$ compact Riemannian manifold, Hu-Shi-Wang \cite%
{hsw} proves that the topological entropy and all Lyapunov exponents of the
central element are zero. These results are special cases of Theorem \ref%
{theorem0} and Corollary \ref{corllary0}, by choosing special length
functions.

We give characterizations of length functions on Lie groups. Our next result
is that there is essentially only one length function on the special linear
group $\mathrm{SL}_{2}(\mathbb{R})$:

\begin{theorem}
\label{th0}Let $G=\mathrm{SL}_{2}(\mathbb{R}).$ Any length function $%
l:G\rightarrow \lbrack 0,+\infty )$ continuous on the subgroup $SO(2)$ and
the diagonal subgroup is proportional to the translation function 
\begin{equation*}
\tau (g):=\inf_{x\in X}d(x,gx),
\end{equation*}%
where $X=\mathrm{SL}_{2}(\mathbb{R})/\mathrm{SO}(2)$ is the upper-half plane.
\end{theorem}

More generally, we study length functions on Lie groups. Let $G$ be a
connected semisimple Lie group whose center is finite with an Iwasawa
decomposition $G=KAN$. Let $W$ be the Weyl group, i.e. the quotient group of
the normalizers $N_{K}(A)$ modulo the centralizers $C_{K}(A)$. Our second
result shows that a length function $l$ on $G$ is uniquely determined by its
image on $A.$

\begin{theorem}
\label{th0.1}Let $G$ be a connected semisimple Lie group whose center is
finite with an Iwasawa decomposition $G=KAN$. Let $W$ be the Weyl group.

\begin{enumerate}
\item[(i)] Any length function $l$ on $G$ that is continuous on the maximal
compact subgroup $K$ is determined by its image on $A.$

\item[(ii)] Conversely, any length function $l$ on $A$ that is $W$-invariant
(i.e. $l(w\cdot a)=l(a)$) can be extended to be a length function on $G$
that vanishes on the maximal compact subgroup $K.$
\end{enumerate}
\end{theorem}

\bigskip

Using the notion of length functions, we treat the rigidity phenomena in a
unified approach. A length function $l:G\rightarrow \lbrack 0,\infty )$ is
called purely positive if $l(g)>0$ for any infinite-order element $g.$ A
group $G$ is called virtually poly-positive, if there is a finite-index
subgroup $H<G$ and a subnormal series%
\begin{equation*}
1=H_{n}\vartriangleleft H_{n-1}\vartriangleleft \cdots \vartriangleleft
H_{0}=H
\end{equation*}%
such that every finitely generated subgroup of each quotient $H_{i}/H_{i+1}$ 
$(i=0,...,n-1)$ has a purely positive length function. Our following results
are on the rigidity of group homomorphisms.

\begin{theorem}
\label{th0.2}Let $\Gamma $ be an arithmetic subgroup of a simple algebraic $%
\mathbb{Q}$-group of $\mathbb{Q}$-rank at least $2.$ Suppose that $G$ is
virtually poly-positive. Then any group homomorphism $f:\Gamma \rightarrow G$
has its image finite.
\end{theorem}

\begin{theorem}
\label{th1}Let $G$ be a group having a finite-index subgroup $H<G$ and a
subnormal series%
\begin{equation*}
1=H_{n}\vartriangleleft H_{n-1}\vartriangleleft \cdots \vartriangleleft
H_{0}=H
\end{equation*}%
satisfying that

\begin{enumerate}
\item[(i)] every finitely generated subgroup of each quotient $H_{i}/H_{i+1}$
$(i=0,...,n-1)$ has a purely positive length function; and

\item[(ii)] any torsion abelian subgroup in every finitely generated
subgroup of each quotient $H_{i}/H_{i+1}$ $(i=0,...,n-1)$ is finitely
generated.
\end{enumerate}

Let $R$ be a finitely generated associative ring with identity and $E_{n}(R)$
the elementary subgroup. Suppose that $\Gamma <E_{n}(R)$ is finite-index
subgroup. Then any group homomorphism $f:\Gamma \rightarrow G$ has its image
finite when $n\geq 3$.
\end{theorem}

\begin{corollary}
\label{mcor}Let $\Gamma $ be an arithmetic subgroup of a simple algebraic $%
\mathbb{Q}$-group of $\mathbb{Q}$-rank at least $2,$ or a finite-index
subgroup of the elementary subgroup $E_{n}(R)$ $(n\geq 3)$ for an
associative ring $R.$ Then any group homomorphism $f:E\rightarrow G$ has its
image finite. Here $G$ is one of the following groups:

\begin{itemize}
\item a Gromov hyperbolic group,

\item $\mathrm{CAT(0)}$ group,

\item automorphism group $\mathrm{Aut}(F_{k})$ of a free group$,$

\item outer automorphism group $\mathrm{Out}(F_{k})$ of a free group,

\item mapping class group $\mathrm{MCG}(\Sigma _{g})$ $(g\geq 2)$, or

\item the group $\mathrm{Diff}(\Sigma ,\omega )$ of diffeomorphisms of a
closed surface preserving an area form $\omega .$
\end{itemize}
\end{corollary}

\begin{theorem}
\label{th2'}Suppose that $G$ is virtually poly-positive. Let $R$ be a
finitely generated associative ring of characteristic zero such that any
nonzero ideal is of a finite index (eg. the ring of algebraic integers in a
number field). Suppose that $S<E_{n}(R)$ is a finite-index subgroup of the
elementary group. Then any group homomorphism $f:S\rightarrow G$ has its
image finite when $n\geq 3$.
\end{theorem}

\begin{corollary}
\label{mor'}Let $R$ be an associative ring of characteristic zero such that
any nonzero ideal is of a finite index. Any group homomorphism $%
f:E\rightarrow G$ has its image finite, where $E<E_{n}(R)$ is finite-index
subgroup and $n\geq 3$. Here $G$ is one of the followings:

\begin{itemize}
\item a $\mathrm{CAT}(0)$ group or more generally a semi-hyperbolic group,

\item a group acting properly semi-simply on a $\mathrm{CAT}(0)$ space, or

\item a group acting properly semi-simply on a $\delta $-hyperbolic space.
\end{itemize}
\end{corollary}

Some relevant cases of Theorem \ref{th1} and Theorem \ref{th2'} are already
established in the literature, whose proofs are usually based on the
distortions in the word length. Bridson and Wade \cite{bw} showed that any
group homomorphism from an irreducible lattice in a semisimple Lie group of
real rank $\geq 2$ to the mapping class group $\mathrm{MCG}(\Sigma _{g})$
must have its image finite. However, Theorem \ref{th0.2} can never hold when 
$\Gamma $ is a cocompact lattice, since a cocompact lattice has its stable
word length purely positive. When the length functions involved in the
virtually poly-positive group $G$ are required to be stable word lengths,
Theorem \ref{th0.2} holds more generally for $\Gamma $ non-uniform
irreducible lattices a semisimple Lie group of real rank $\geq 2$ (see
Proposition \ref{polyword}). When the length functions involved in the
virtually poly-positive group $G$ are given by a particular kind of
quasi-cocycles, Theorem \ref{th0.2} holds more generally for $\Gamma $ with
property TT (cf. Py \cite{py}, Prop. 2.2). Haettel \cite{ha} proves that any
action of a high-rank lattice on a Gromov-hyperbolic space is elementary
(i.e. either elliptic or parabolic). Guirardel and Horbez \cite{ha} prove
that every group homomorphism from a high-rank lattice to the outer
automorphism group of the torsion-free hyperbolic group has a finite image.
Thom \cite{thom} (Corollary 4.5) proves that any group homomorphism from a
boundedly generated group with property T to a Gromov hyperbolic group has a
finite image. Compared with these results, our target group $G$ and the
source group $E_{n}(R)$ (can be defined over any non-commutative ring) in
Theorem \ref{th1} are much more general. The inequalities of $n$ in Theorem %
\ref{th1}, Theorem \ref{th2'} and Corollary \ref{mcor}, Corollary \ref{mor'}
can not be improved, since $\mathrm{SL}_{2}(\mathbb{Z})$ is hyperbolic. The
group $\Gamma $ in Corollary \ref{mcor} has Kazhdan's property T (i.e. an
arithmetic subgroup of a simple algebraic $\mathbb{Q}$-group of $\mathbb{Q}$%
-rank at least $2,$ or a finite-index subgroup of the elementary subgroup $%
E_{n}(R),$ $n\geq 3,$ for an associative ring $R$ has Kazhdan's property T
by \cite{ej}). However, there exist hyperbolic groups with Kazhdan's
property $T$ (cf. \cite{gromov}, Section 5.6). This implies that Corollary %
\ref{mcor} does not hold generally for groups $\Gamma $ with Kazhdan's
property T. Franks and Handel \cite{fh} prove that any group homomorphism
from a quasi-simple group containing a subgroup isomorphic to the
three-dimensional integer Heisenberg group, to the group $\mathrm{Diff}%
(\Sigma ,\omega )$ of diffeomorphisms of a closed surface preserving an area
form $\omega ,$ has its image finite (cf. Lemma \ref{lemaheis}).

\bigskip

We now study length functions on the Cremona groups.

\begin{theorem}
\label{th3}Let $\mathrm{Bir}(\mathbb{P}_{k}^{n})$ $(n\geq 2)$ be the group
of birational maps on the projective space $\mathbb{P}_{k}^{n}$ over an
algebraically closed field $k$. Any length function $l:\mathrm{Bir}(\mathbb{P%
}_{k}^{n})\rightarrow \lbrack 0,+\infty )$ vanishes on the automorphism
group $\mathrm{Aut}(\mathbb{P}_{k}^{n})=\mathrm{PGL}_{n+1}(k).$
\end{theorem}

When $n=2,$ a result of Blanc and Furter \cite{bd} (page 193 and Proposition
4.41) implies that there are three length functions $l_{1},l_{2},l_{3}$ on $%
\mathrm{Bir}(\mathbb{P}_{k}^{2})$ such that any element $g\in \mathrm{Bir}(%
\mathbb{P}_{k}^{n})$ satisfying $l_{1}(g)=l_{2}(g)=l_{3}(g)=0$ is either
finite or conjugate to an element in $\mathrm{Aut}(\mathbb{P}_{k}^{2}).$
This implies that the automorphism group $\mathrm{Aut}(\mathbb{P}_{k}^{n})$
(when $k=2$) is one of the `largest' subgroups of $\mathrm{Bir}(\mathbb{P}%
_{k}^{n})$ on which every length function vanishes. The distortion of
elements in the word metric is completely obtained by Cantat and Cornulier 
\cite{cc}.

\begin{corollary}
\label{fincor}Let $G$ be a virtually poly-positive group. Any group
homomorphism $f:\mathrm{Bir}(\mathbb{P}_{k}^{2})\rightarrow G$ is trivial,
for an algebraically closed field $k$.
\end{corollary}

In particular, Corollary \ref{fincor} implies that any quotient group of $%
\mathrm{Bir}(\mathbb{P}_{k}^{2})$ can act properly semisimply neither on a
Gromov $\delta $-hyperbolic space nor a $\mathrm{CAT}(0)$ space. This is
interesting, considering the following facts. There are
(infinite-dimensional) hyperbolic space and cubical complexes, on which $%
\mathrm{Bir}(P_{k}^{2})$ acts isometrically (see \cite{can}, Section 3.1.2
and \cite{lr}). The Cremona group $\mathrm{Bir}(\mathbb{P}_{k}^{2})$ is
sub-quotient universal: every countable group can be embedded in a quotient
group of $\mathrm{Bir}(\mathbb{P}_{k}^{2})$ (see \cite{can}, Theorem 4.7).
Moreover, Blanc-Lamy-Zimmermann \cite{blz} (Theorem E) proves that when $%
n\geq 3,$ there is a surjection from $\mathrm{Bir}(\mathbb{P}_{k}^{n})$ onto
a free product of two-element groups $\mathbb{Z}/2.$ This means that
Corollary \ref{fincor} can never hold for higher dimensional Cremona groups.

\bigskip

The proofs of Theorems \ref{th0} and \ref{th0.1} are based the
Jordan-Chevalley decompositions of algebraic groups and Lie groups. We will
prove that any length function on a Heisenberg group vanishes on the central
elements (see Lemma \ref{nilpfac}). This is a key step for many other
proofs. Based on this fact, we prove Theorems \ref{th0.2}, \ref{th2'}, \ref%
{th1} and \ref{th3} by looking for Heisenberg subgroups. In Section 1, we
give some elementary facts on the length functions. In Section 2, we discuss
typical examples of length functions. In later sections, we study length
functions on Lie groups, algebraic groups, hyperbolic groups, matrix groups
and the Cremona groups. Theorem \ref{theorem0} and Corollary \ref{corllary0}
are proved in Section \ref{section4}. Theorem \ref{th0} is proved in Section %
\ref{section6}. Theorem \ref{th0.1} is proved in Section \ref{section7}.
Theorem \ref{th0.2} is proved in Section \ref{section8}, while Theorems \ref%
{th1} and \ref{th2'}, Corollaries \ref{mcor} and \ref{mor'} are proved in
Section \ref{section9}. Theorem \ref{th3} and Corollary \ref{fincor} are
proved in Section \ref{section10}.

\section{Basic properties of length functions}

\subsection{Length functions}

\begin{definition}
\label{def}Let $G$ be a group. A function $l:G\rightarrow \lbrack 0,\infty )$
is called a length function if

1) $l(g^{n})=|n|l(g)$ for any $g\in G$ and $n\in \mathbb{Z}.$

2) $l(hgh^{-1})=l(g)$ for any $h,g\in G.$

3) $l(ab)\leq l(a)+l(b)$ for \emph{commuting} elements $a,b\in G,$ i.e. $%
ab=ba.$
\end{definition}

\begin{lemma}
\label{1.2}Any torsion element $g\in G$ has length $l(g)=0.$
\end{lemma}

\begin{proof}
Note that $l(1)=2l(1)$ and thus $l(1)=0.$ If $g^{n}=1,$ then $%
l(g)=l(1)/|n|=0.$
\end{proof}

Recall that a subset $V$ of a real vector space is a convex cone, if $%
av+bw\in V$ for any $v,w\in V$ and any non-negative real numbers $a,b\geq 0.$

\begin{lemma}
The set $\mathrm{Func}(G)$ of all length functions on a group $G$ is a
convex cone.
\end{lemma}

\begin{proof}
It is obvious that for two functions $l_{1},l_{2}$ on $G,$ a non-negative
linear combination $al_{1}+bl_{2}$ is a new length function.
\end{proof}

\begin{lemma}
\label{lemq}Let $f:G\rightarrow H$ be a group homomorphism between two
groups $G$ and $H.$ For any length function $l:H\rightarrow \lbrack 0,\infty
),$ the composite $l\circ f$ is a length function on $G.$
\end{lemma}

\begin{proof}
It is enough to note that a group homomorphism preserves the powers of
elements, conjugacy classes and commutativity of elements.
\end{proof}

\begin{corollary}
For a group $G,$ let $\mathrm{Out}(G)=\mathrm{Aut}(G)/\mathrm{Inn}(G)$ be
the outer automorphism group. Then $\mathrm{Out}(G)$ acts on the set $%
\mathrm{Func}(G)$ of all length functions by pre-compositions 
\begin{equation*}
l\mapsto l\circ g,
\end{equation*}
where $l\in \mathrm{Func}(G),$ $g\in \mathrm{Out}(G)$. This action preserves
scalar multiplications and linear combinations (with non-negative
coefficients).
\end{corollary}

\begin{proof}
For an inner automorphism $I_{g}:G\rightarrow G$ given by $%
I_{g}(h)=ghg^{-1}, $ the length function $l\circ I_{g}=l$ since $l$ is
invariant under conjugation. Therefore, the outer automorphism group $%
\mathrm{Out}(G)$ has an action on $\mathrm{Func}(G).$ It is obvious that the
pre-compositions preserve scalar multiplications and linear combinations
with non-negative coefficients.
\end{proof}

\begin{definition}
A length function $l:G\rightarrow \lbrack 0,\infty )$ is primitive if it is
not a composite $l^{\prime }\circ f$ for a non-trivial surjective group
homomorphism $f:G\twoheadrightarrow H$ and a length function $l^{\prime }:$ $%
H\rightarrow \lbrack 0,\infty ).$
\end{definition}

\begin{lemma}
\label{1.7}Suppose that a length function $l:G\rightarrow \lbrack 0,\infty )$
vanishes on a central subgroup $H<G.$ Then $l$ factors through the quotient
group $G/H.$ In other words, there exists a homogeneous and
conjugate-invariant function $l^{\prime }:G/H\rightarrow \lbrack 0,\infty )$
such that $l=l^{\prime }\circ q,$ where $q:G\rightarrow G/H$ is the quotient
group homomorphism.
\end{lemma}

\begin{proof}
Write $G=\cup gH,$ the union of left cosets. For any $h\in H,$ we have $%
l(gh)\leq l(g)+l(h)=l(g)$ and $l(g)=l(ghh^{-1})\leq l(gh).$ Therefore, $%
l(gh)=l(g)$ for any $h\in H.$ Define $l^{\prime }(gH)=l(g).$ Then $l^{\prime
}$ is a function on the quotient group $G/H.$ The required property follows
easily.
\end{proof}

\begin{corollary}
\label{cor1.1}Suppose that a group $G$ has a non-trivial finite central
subgroup $Z(G).$ Any length function $l$ on $G$ factors through $G/Z(G).$
\end{corollary}

\begin{proof}
This follows from Lemma \ref{1.7} and Lemma \ref{1.2}.
\end{proof}

\begin{lemma}
\label{lem1.2}Let $G$ be a group. Suppose that any non-trivial normal
subgroup $H\vartriangleleft G$ is of finite index. Then any non-vanishing
length function $l:G\rightarrow \lbrack 0,\infty )$ is primitive.
\end{lemma}

\begin{proof}
Suppose that $l$ is a composite $l^{\prime }\circ f$ for a non-trivial
surjective group homomorphism $f:G\twoheadrightarrow Q$ and a length
function $l^{\prime }:Q\rightarrow \lbrack 0,\infty ).$ By the assumption of 
$G,$ the quotient group $Q$ is finite. This implies that $l^{\prime }$ and
thus $l$ vanishes, which is a contradiction.
\end{proof}

\begin{corollary}
Let $\Gamma $ be an irreducible lattice in a connected irreducible
semisimple Lie group of real rank $\geq 2.$ Then any non-vanishing length
function $l:\Gamma /Z(\Gamma )\rightarrow \lbrack 0,\infty )$ is primitive.
\end{corollary}

\begin{proof}
By the Margulis-Kazhdan theorem (see \cite{z84}, Theorem 8.1.2), any normal
subgroup $N$ of $\Gamma $ either lies in the center of $\Gamma $ (and hence
it is finite) or the quotient group $\Gamma /N$ is finite. The previous
Lemma \ref{lem1.2} implies that $l$ is primitive.
\end{proof}

\section{Examples of length functions\label{section3}}

Let's see a general example first. Let $G$ be a group and $f:G\rightarrow
\lbrack 0,+\infty )$ be a function satisfying $f(gh)\leq f(g)+f(h)$ and $%
f(g)=f(g^{-1})$ for any elements $g,h\in G.$ Define $l:G\rightarrow \lbrack
0,+\infty )$ by 
\begin{equation*}
l(g)=\lim_{n\rightarrow \infty }\frac{f(g^{n})}{n}
\end{equation*}%
for any $g\in G.$

\begin{lemma}
\label{lem0}The function $l$ is a length function in the sense of Definition %
\ref{def}.
\end{lemma}

\begin{proof}
For any $g\in G,$ and natural numbers $n,m,$ we have $f(g^{n+m})\leq
f(g^{n})+f(g^{m}).$ This means that $\{f(g^{n})\}_{n=1}^{\infty }$ is a
subadditive sequence and thus the limit $\lim_{n\rightarrow \infty }\frac{%
f(g^{n})}{n}$ exists. This shows that $l$ is well-defined.

From the definition of $l,$ it is clear that $l(g^{n})=|n|l(g)$ for any
integer $n.$ Let $h\in G.$ We have%
\begin{equation*}
l(hgh^{-1})=\lim_{n\rightarrow \infty }\frac{f(hg^{n}h^{-1})}{n}\leq
\lim_{n\rightarrow \infty }\frac{f(h)+f(g^{n})+f(h^{-1})}{n}%
=\lim_{n\rightarrow \infty }\frac{f(g^{n})}{n}=l(g).
\end{equation*}%
Similarly, we have $l(g)=l(h^{-1}(hgh^{-1})h)\leq l(hgh^{-1})$ and thus $%
l(g)=l(hgh^{-1}).$ For commuting elements $a,b,$ we have $%
(ab)^{n}=a^{n}b^{n}.$ Therefore, 
\begin{eqnarray*}
l(ab) &=&\lim_{n\rightarrow \infty }\frac{f((ab)^{n})}{n}=\lim_{n\rightarrow
\infty }\frac{f(a^{n}b^{n})}{n} \\
&\leq &\lim_{n\rightarrow \infty }\frac{f(a^{n})+f(b^{n})}{n}\leq l(a)+l(b).
\end{eqnarray*}
\end{proof}

Many (but not all) length functions $l$ come from these subadditive
functions $f.$

\subsection{Stable word lengths}

Let $G$ be a group generated by a (not necessarily finite) set $S$
satisfying $s^{-1}\in S$ for each $s\in S.$ For any $g\in G,$ the word
length $\phi _{S}(g)=\min \{n\mid g=s_{1}s_{2}\cdots s_{n},$ each $s_{i}\in
S\}$ is the minimal number of elements of $S$ whose product is $g.$ The
stable length $l(g)=\lim_{n\rightarrow \infty }\frac{\phi _{S}(g^{n})}{n}.$
Since $\phi _{S}(g^{n})$ is subadditive, the limit always exists.

\begin{lemma}
The stable length $l:G\rightarrow \lbrack 0,+\infty )$ is a length function
in the sense of Definition \ref{def}.
\end{lemma}

\begin{proof}
From the definition of the word length $\phi _{S},$ it is clear that $\phi
_{S}(gh)\leq \phi _{S}(g)+\phi _{S}(h)$ and $\phi _{S}(g)=\phi _{S}(g^{-1})$
for any $g,h\in G.$ The claim is proved by Lemma \ref{lem0}.
\end{proof}

\bigskip

When $S$ is the set of commutators, the $l(g)$ is called the stable
commutator length, which is related to lots of topics in low-dimensional
topology (see Calegari \cite{cal}).

\subsection{Growth rate}

Let $G$ be a group generated by a finite set $S$ satisfying $s^{-1}\in S$
for each $s\in S.$ Suppose $|\cdot |_{S}$ is the word length of $(G,S).$ For
any automorphism $\alpha :G\rightarrow G,$ define $l^{\prime }(\alpha )=\max
\{|\alpha (s_{i})|_{S}:s_{i}\in S\}.$ Let $l(\alpha )=\lim_{n\rightarrow
\infty }\frac{\log l^{\prime }(\alpha ^{n})}{n}.$ This number $l(\alpha )$
is called the algebraic entropy of $\alpha $ (cf. \cite{kat}, Definition
3.1.9, page 114).

\begin{lemma}
Let $\mathrm{Aut}(G)$ be the group of automorphisms of $G.$ The function $l:%
\mathrm{Aut}(G)\rightarrow \lbrack 0,+\infty )$ is a length function in the
sense of Definition \ref{def}.
\end{lemma}

\begin{proof}
Since $\alpha (s_{i})^{-1}=\alpha ^{-1}(s_{i})$ for any $s_{i}\in S,$ we
know that $l^{\prime }(\alpha )=l^{\prime }(\alpha ^{-1}).$ For another
automorphism $\beta :G\rightarrow G,$ let $l^{\prime }(\beta )=|\beta
(s_{i})|_{S}$ for some $s_{i}\in S.$ Suppose that $\beta
(s_{i})=s_{i_{1}}s_{i_{2}}\cdots s_{i_{k}}$ with $k=l^{\prime }(\beta ).$
Then $|(\alpha \beta )(s_{i})|_{S}=|\alpha (s_{i_{1}})\alpha
(s_{i_{2}})\cdots \alpha (s_{i_{k}})|_{S}\leq l^{\prime }(\alpha )k.$ This
proves that $l^{\prime }(\alpha \beta )\leq l^{\prime }(\alpha )l^{\prime
}(\beta ).$ The claim is proved by Lemma \ref{lem0}.
\end{proof}

\bigskip

Fix $g\in G.$ For any automorphism $\alpha :G\rightarrow G,$ define $%
b_{n}=|\alpha ^{n}(g)|_{S}.$ Suppose that $g=s_{1}s_{2}\cdots s_{k}$ with $%
k=|g|_{S}.$ Note that $b_{n}=|\alpha ^{n}(g)|_{S}=|\alpha ^{n}(s_{1})\alpha
^{n}(s_{2})\cdots \alpha ^{n}(s_{k})|_{S}\leq l^{\prime }(\alpha
^{n})|g|_{S}.$ Therefore, we have 
\begin{equation*}
\lim \sup_{n\rightarrow \infty }\frac{\log b_{n}}{n}\leq l(\alpha ).
\end{equation*}%
This implies that $l(\alpha )$ is an upper bound for the growth rate of $%
\{|\alpha ^{n}(g)|_{S}\}.$ The growth rate is studied a lot in geometric
group theory (for example, see \cite{le} for growth of automorphisms of free
groups).

\subsection{Matrix norms and group acting on smooth manifolds}

For a square matrix $A,$ the matrix norm $\Vert A\Vert =\sup_{\Vert x\Vert
=1}\Vert Ax\Vert .$ Define the stable norm $s(A)=\lim_{n\rightarrow +\infty }%
\frac{\log \Vert A^{n}\Vert }{n}.$ Since $\Vert AB\Vert \leq \Vert A\Vert
\Vert B\Vert $ for any two matrices $A,B,$ the sequence $\{\log \Vert
A^{n}\Vert \}_{n=1}^{\infty }$ is subadditive and thus the limit exists.

\begin{lemma}
Let $G=\mathrm{GL}_{n}(\mathbb{R})$ be the general linear group. The
function $l:G\rightarrow \lbrack 0,+\infty )$ defined by%
\begin{equation*}
l(g)=\max \{s(g),s(g^{-1})\}
\end{equation*}
is a length function in the sense of Definition \ref{def}.
\end{lemma}

\begin{proof}
From the definition of the matrix norm$,$ it is clear that $\log \Vert
gh\Vert \leq \log \Vert g\Vert +\log \Vert h\Vert $ for any $g,h\in G.$ Then 
$l(g)=\max \{s(g),s(g^{-1})\}$ is a length function by Lemma \ref{lem0}.
\end{proof}

Let $M$ be a compact Riemannian manifold and $\mathrm{Diff}(M)$ the
diffeomorphism group consisting of all self-diffeomorphisms. For any
diffeomorphism $f:M\rightarrow M,$ let 
\begin{equation*}
\Vert f\Vert =\sup_{x\in M}\Vert D_{x}f\Vert ,
\end{equation*}%
where $D_{x}f$ is the induced linear map between tangent spaces $%
T_{x}M\rightarrow T_{f(x)}M.$ Define 
\begin{equation*}
l(f)=\max \{\lim_{n\rightarrow +\infty }\frac{\log \Vert f^{n}\Vert }{n}%
,\lim_{n\rightarrow +\infty }\frac{\log \Vert f^{-n}\Vert }{n}\}.
\end{equation*}%
A similar argument as the proof of the previous lemma proves the following.

\begin{lemma}
\label{derivelength}Let $G$ be a group acting on a Riemannian manifold $M$
by diffeomorphisms. The function $l:G\rightarrow \lbrack 0,+\infty )$ is a
length function in the sense of Definition \ref{def}.
\end{lemma}

For an $f$-invariant Borel probability measure $\mu $ on $M,$ it is well
known (see \cite{O}) that there exists a measurable subset $\Gamma
_{f}\subset M$ with $\mu (\Gamma _{f})=1$ such that for all $x\in \Gamma
_{f} $ and $u\in T_{x}M,$ the limit%
\begin{equation*}
\chi (x,u,f)=\lim \frac{1}{n}\log \Vert D_{x}f^{n}(u)\Vert
\end{equation*}%
exists and is called the Lyapunov exponent of $u$ at $x.$ From the
definitions, we know that $\chi (x,u,f)\leq l(f)$ for any $x\in \Gamma _{f}$
and $u\in T_{x}M.$

\subsection{Smooth measure-theoretic entropy}

Let $T:X\rightarrow X$ be a measure-preserving map of the probability space $%
(X,\mathfrak{B},m).$ For a finite-sub-$\sigma $-algebra $A=%
\{A_{1},A_{2},...,A_{k}\}$ of $\mathfrak{B},$ denote by%
\begin{eqnarray*}
H(A) &=&-\sum m(A_{i})\log m(A_{i}), \\
h(T,A) &=&\lim \frac{1}{n}H(\vee _{i=0}^{n-1}T^{-i}A),
\end{eqnarray*}%
where $\vee _{i=0}^{n-1}T^{-i}A$ is a set consisting of sets of the form $%
\cap _{i=0}^{n-1}T^{-i}A_{j_{i}}.$ The entropy of $T$ is defined as $%
h_{m}(T)=\sup h(T,A),$ where the supremum is taken over all finite
sub-algebra $A$ of $\mathfrak{B}.$ For more details, see Walters \cite{wat}
(Section 4.4).

\begin{lemma}
Let $M$ be a $C^{\infty }$ closed Riemannian manifold and $G=\mathrm{Diff}%
_{\mu }^{2}(M)$ consisting of diffeomorphisms of $M$ preserving a Borel
probability measure $\mu .$ The entropy $h_{\mu }$ is a length function on $%
\mathrm{Diff}_{\mu }^{2}(M)$ in the sense of Definition \ref{def}.
\end{lemma}

\begin{proof}
For any $f,g\in \mathrm{Diff}_{\mu }^{2}(M)$ and integer $n,$ it is
well-known that $h_{\mu }(f^{n})=|n|h_{\mu }(f)$ and $h_{\mu }(f)=h_{\mu
}(gfg^{-1})$ (cf. \cite{wat}, Theorem 4.11 and Theorem 4.13). Hu \cite{hu}
proves that $h_{\mu }(fg)\leq h_{\mu }(f)+h_{\mu }(g)$ when $fg=gf.$
\end{proof}

\subsection{Stable translation length on metric spaces}

Let $(X,d)$ be a metric space and $\gamma :X\rightarrow X$ an isometry. Fix $%
x\in X.$ Note that $d(x,\gamma _{1}\gamma _{2}x)\leq d(x,\gamma
_{1}x)+d(\gamma _{1}x,\gamma _{1}\gamma _{2}x)=d(x,\gamma _{1}x)+d(x,\gamma
_{2}x)$ and $d(x,\gamma _{1}x)=d(x,\gamma _{1}^{-1}x)$ for any isometries $%
\gamma _{1},\gamma _{2}.$ Define%
\begin{equation*}
l(\gamma )=\lim_{n\rightarrow \infty }\frac{d(x,\gamma ^{n}x)}{n}.
\end{equation*}%
For any $y\in X,$ we have 
\begin{eqnarray*}
d(x,\gamma ^{n}x) &\leq &d(x,y)+d(y,\gamma ^{n}y)+d(\gamma ^{n}y,\gamma
^{n}x) \\
&=&2d(x,y)+d(y,\gamma ^{n}y)
\end{eqnarray*}%
and thus $\lim_{n\rightarrow \infty }\frac{d(x,\gamma ^{n}x)}{n}\leq
\lim_{n\rightarrow \infty }\frac{d(y,\gamma ^{n}y)}{n}.$ Similarly, we have
the other direction 
\begin{equation*}
\lim_{n\rightarrow \infty }\frac{d(y,\gamma ^{n}y)}{n}\leq
\lim_{n\rightarrow \infty }\frac{d(x,\gamma ^{n}x)}{n}.
\end{equation*}
This shows that the definition of $l(\gamma )$ does not depend on the choice
of $x.$

\begin{lemma}
\label{trans}Let $G$ be a group acting isometrically on a metric space $X.$
Then the function $l:G\rightarrow \lbrack 0,+\infty )$ defined by $%
g\longmapsto l(g)$ as above is a length function in the sense of Definition %
\ref{def}.
\end{lemma}

\begin{proof}
This follows from Lemma \ref{lem0}.
\end{proof}

\subsection{Translation lengths of isometries of CAT(0) spaces}

In this subsection, we will prove that the translation length on a CAT(0)
space defines a length function. First, let us introduce some notations. Let 
$(X,d_{X})$ be a geodesic metric space, i.e. any two points $x,y\in X$ can
be connected by a path $[x,y]$ of length $d_{X}(x,y)$. For three points $%
x,y,z\in X,$ the geodesic triangle $\Delta (x,y,z)$ consists of the three
vertices $x,y,z$ and the three geodesics $[x,y],[y,z]$ and $[z,x].$ Let $%
\mathbb{R}^{2}$ be the Euclidean plane with the standard distance $d_{%
\mathbb{R}^{2}}$ and $\bar{\Delta}$ a triangle in $\mathbb{R}^{2}$ with the
same edge lengths as $\Delta $. Denote by $\varphi :\Delta \rightarrow \bar{%
\Delta}$ the map sending each edge of $\Delta $ to the corresponding edge of 
$\bar{\Delta}.$ The space $X$ is called a \textrm{CAT(0)} space if for any
triangle $\Delta $ and two elements $a,b\in \Delta ,$ we have the inequality 
\begin{equation*}
d_{X}(a,b)\leq d_{\mathbb{R}^{2}}(\varphi (a),\varphi (b)).
\end{equation*}%
The typical examples of \textrm{CAT(0)} spaces include simplicial trees,
hyperbolic spaces $\mathbb{H}^{n}$, products of \textrm{CAT(0) }spaces and
so on. From now on, we assume that $X$ is a complete \textrm{CAT(0)} space.
Denote by \textrm{Isom}$(X)$ the isometry group of $X.$ For any $g\in $ 
\textrm{Isom}$(X)$\textrm{, }let 
\begin{equation*}
\mathrm{Minset}(g)=\{x\in X:d(x,gx)\leq d(y,gy)\text{ for any }y\in X\}
\end{equation*}%
and let $\tau (g)=\inf\nolimits_{x\in X}d(x,gx)$ be the translation length
of $g.$ When the fixed-point set $\mathrm{Fix}(g)\neq \emptyset ,$ we call $%
g $ elliptic. When $\mathrm{Minset}(g)\neq \emptyset $ and $d_{X}(x,gx)=\tau
(g)>0$ for any $x\in \mathrm{Minset}(g),$ we call $g$ hyperbolic. The group
element $g$ is called semisimple if the minimal set $\mathrm{Minset}(g)$ is
not empty, i.e. it is either elliptic or hyperbolic. A subset $C$ of a 
\textrm{CAT(0)} space is convex, if any two points $x,y\in C$ can connected
by the geodesic segment $[x,y]\subset C.$ A group $G$ is called \textrm{%
CAT(0) if }$G$ acts properly discontinuously and cocompactly on a \textrm{%
CAT(0) }space $X$. In such a case, any infinite-order element in $G$ acts
hyperbolically on $X.$ For more details on \textrm{CAT(0) }spaces, see the
book of Bridson and Haefliger \cite{bh}.

The following was proved by Ballmann-Gromov-Schroeder \cite{bg} (Lemma 6.6,
page 83). The original proof was for Hardmard manifolds, which is also
(well-known) holding for general cases. For completeness, we give details
here.

\begin{lemma}
\label{bgs}Let $\gamma :X\rightarrow X$ be an isometry of a complete \textrm{%
CAT(0) }space $X.$ For any $x_{0}\in X,$ we have 
\begin{equation*}
\tau (\gamma ):=\inf_{x\in X}d(\gamma x,x)=\lim_{k\rightarrow \infty }\frac{%
d(\gamma ^{k}x_{0},x_{0})}{k}.
\end{equation*}
\end{lemma}

\begin{proof}
For any $p=x_{0}\in X,$ let $m$ be the middle point of $[p,\gamma p].$ We
have that $d(m,\gamma m)\leq \frac{1}{2}d(p,\gamma ^{2}p)$ by the convexity
of length functions. Therefore, $d(p,\gamma ^{2}p)\geq 2\tau (\gamma )$ and $%
\tau (\gamma ^{2})\geq 2\tau (\gamma ).$ Note that $d(p,\gamma ^{2}p)\leq
d(p,\gamma p)+d(\gamma p,\gamma ^{2}p)=2d(p,\gamma p)$ and thus $\tau
(\gamma ^{2})\leq 2\tau (\gamma ).$ Inductively, we have 
\begin{equation*}
2^{n}\tau (\gamma )\leq d(p,\gamma ^{2^{n}}p)\leq 2^{n}d(p,\gamma p).
\end{equation*}
Note that the limit $\lim_{k\rightarrow \infty }\frac{d(\gamma ^{k}p,p)}{k}$
exists and is independent of $p$ (see the previous subsection). Therefore,
the limit $\lim_{k\rightarrow \infty }\frac{d(\gamma ^{k}p,p)}{k}$ equals to 
$\tau (\gamma ).$
\end{proof}

\begin{corollary}
\label{procat}Let $X$ be a complete \textrm{CAT(0)} space and $G$ a group
acting on $X$ by isometries. For any $g\in G,$ define $\tau (g)=\inf_{x\in
X}d(x,gx)$ as the translation length. Then $\tau :G\rightarrow \lbrack
0,+\infty )$ is a length function in the sense of Definition \ref{def}.
\end{corollary}

\begin{proof}
This follows from Lemma \ref{bgs} and Lemma \ref{trans}.
\end{proof}

\subsection{Translation lengths of Gromov $\protect\delta $-hyperbolic spaces%
}

Let $\delta >0.$ A geodesic metric space $X$ is called Gromov $\delta $%
-hyperbolic if for any geodesic triangle $\Delta xyz$ one side $[x,y]$ is
contained a $\delta $-neighborhood of the other two edges $[x,z]\cup \lbrack
y,z].$ Fix $x_{0}\in X.$ Any isometry $\gamma :X\rightarrow X$ is called
elliptic if $\{\gamma ^{n}x\}_{ne\mathbb{Z}}$ is bounded. If the orbit map $%
\mathbb{Z}\rightarrow X$ given by $n\mapsto \gamma ^{n}x_{0}$ is
quasi-isometric (i.e. there exists $A\geq 1$ and $B\geq 0$ such that 
\begin{equation*}
\frac{1}{A}|n-m|-B\leq d_{X}(\gamma ^{n}x_{0},\gamma ^{m}x_{0})\leq A|n-m|+B
\end{equation*}%
for any integers $n,m$), we call that $\gamma $ is hyperbolic. Otherwise, we
call that $\gamma $ is parabolic. Define $l(\gamma )=\lim_{n\rightarrow
\infty }\frac{d(\gamma ^{n}x_{0},x_{0})}{n}.$ For any group $G$ acts
isometrically on a $\delta $-hyperbolic space, the function $l:G\rightarrow
\lbrack 0,\infty )$ is a length function by Lemma \ref{trans}. A finitely
generated group $G$ is Gromov $\delta $-hyperbolic if for some finite
generating set $S,$ the Caley graph $\Gamma (G,S)$ is Gromov $\delta $%
-hyperbolic. Any infinite-order element $g$ in a Gromov $\delta $-hyperbolic
group is hyperbolic and thus has positive length $l(g)>0$ (cf. \cite{gromov}%
, 8.1.D). For more details on hyperbolic spaces and hyperbolic groups, see
the book \cite{gromov} of Gromov.

\subsection{Quasi-cocycles}

Let $G$ be a group and $(E,\Vert \cdot \Vert )$ be a normed vector space
with an $G$-action by linear isometries. A function $f:G\rightarrow E$ is a
quasi-cocycle if there exists $C>0$ such that 
\begin{equation*}
\Vert f(gh)-f(g)-gf(h)\Vert <C
\end{equation*}%
for any $g,h\in G.$ Let $l:G\rightarrow \lbrack 0,+\infty )$ be defined by 
\begin{equation*}
l(g)=\lim_{n\rightarrow \infty }\frac{\Vert f(g^{n})\Vert }{n}.
\end{equation*}%
Note that $\Vert f(g^{n+m})\Vert \leq \Vert f(g^{n})\Vert +\Vert
f(g^{m})\Vert +C$ for any integers $n,m\geq 0.$ This general subadditive
property implies that the limit $\lim_{n\rightarrow \infty }\frac{\Vert
f(g^{n})\Vert }{n}$ exists (see \cite{ste}, Theorem 1.9.2, page 22). We call 
$l$ the average norm. Many applications of quasi-cocycles can be found in 
\cite{monod}.

\begin{lemma}
For any quasi-cocycle $f:G\rightarrow E,$ the average norm $l$ is a length
function.
\end{lemma}

\begin{proof}
For any natural number $n,$ we have%
\begin{equation*}
\Vert f(1)-f(g^{-n})-g^{-n}f(g^{n})\Vert <C
\end{equation*}%
and thus $\Vert \frac{f(1)-f(g^{-n})-g^{-n}f(g^{n})}{n}\Vert <\frac{C}{n}.$
Taking the limit, we have $\lim_{n\rightarrow \infty }\frac{\Vert
f(g^{-n})\Vert }{n}=\lim_{n\rightarrow \infty }\frac{\Vert f(g^{n})\Vert }{n}%
.$ Therefore, for any $k\in \mathbb{Z},$ we have $l(g^{k})=\lim_{n%
\rightarrow \infty }\frac{\Vert f(g^{kn})\Vert }{n}=|k|l(g).$ For any $h\in
G,$ we have 
\begin{equation*}
\Vert f(hg^{n}h^{-1})\Vert \leq \Vert f(h)\Vert +\Vert f(h^{-1})\Vert +\Vert
f(g^{n})\Vert +2C.
\end{equation*}%
Therefore, we have 
\begin{equation*}
l(hgh^{-1})=\lim_{n\rightarrow \infty }\frac{\Vert f(hg^{n}h^{-1})\Vert }{n}%
\leq \lim_{n\rightarrow \infty }\frac{\Vert f(g^{n})\Vert }{n}=l(g).
\end{equation*}%
Similarly, we have $l(g)=l(h^{-1}(hgh^{-1})h)\leq l(hgh^{-1}).$ When $g,h$
commutes, we have%
\begin{eqnarray*}
l(gh) &=&\lim_{n\rightarrow \infty }\frac{\Vert f((gh)^{n})\Vert }{n}%
=\lim_{n\rightarrow \infty }\frac{\Vert f(g^{n}h^{n})\Vert }{n} \\
&\leq &\lim_{n\rightarrow \infty }\frac{\Vert f(g^{n})\Vert +\Vert
f(h^{n})\Vert +C}{n}=l(g)+l(h).
\end{eqnarray*}
\end{proof}

\subsection{Rotation number}

Let $\mathbb{R}$ be the real line and $\mathrm{Homeo}_{\mathbb{Z}}(\mathbb{R}%
)=\{f\mid f:\mathbb{R}\rightarrow \mathbb{R}$ is a monotonically increasing
homeomorphism such that $f(x+n)=f(x)+n$ for any $n\in \mathbb{Z}\}.$ For any 
$f\in \mathrm{Homeo}_{\mathbb{Z}}(\mathbb{R})$ and $x\in \lbrack 0,1),$ the
translation number is defined as 
\begin{equation*}
l(f)=\lim_{n\rightarrow \infty }\frac{f^{n}(x)-x}{n}.
\end{equation*}%
It is well-known that $l(f)$ exists and is independent of $x$ (see \cite{na}%
, Prop. 2.22, p.31). Note that every $f\in \mathrm{Homeo}_{\mathbb{Z}}(%
\mathbb{R})$ induces an orientation-preserving homeomorphism of the circle $%
S^{1}.$

\begin{proposition}
The absolute value of the translation number $|l|:\mathrm{Homeo}_{\mathbb{Z}%
}(\mathbb{R})\rightarrow \lbrack 0,\infty )$ is a length function in the
sense of Definition \ref{def}.
\end{proposition}

\begin{proof}
For any $f\in \mathrm{Homeo}_{\mathbb{Z}}(\mathbb{R})$ and $k\in \mathbb{Z}%
\backslash \{0\},$ we have that 
\begin{equation*}
l(f^{k})=\lim_{n\rightarrow \infty }\frac{f^{kn}(x)-x}{n}=k\lim_{n%
\rightarrow \infty }\frac{f^{kn}(x)-x}{nk}=kl(f).
\end{equation*}%
For any $a\in \mathrm{Homeo}_{\mathbb{Z}}(\mathbb{R}),$ we have that%
\begin{eqnarray*}
&\mid &l(afa^{-1})-l(f)\mid =\lim_{n\rightarrow \infty }\mid \frac{%
af^{n}(a^{-1}x)-x-f^{n}(x)+x}{n}\mid \\
&=&\lim_{n\rightarrow \infty }\mid \frac{%
af^{n}(a^{-1}x)-f^{n}(a^{-1}x)+f^{n}(a^{-1}x)-f^{n}(x)}{n}\mid \\
&=&0,
\end{eqnarray*}%
since $a-\mathrm{id}_{\mathbb{R}}$ is bounded on $[0,1]$ and $\mid
f^{n}(a^{-1}x)-f^{n}(x)\mid \leq 2+|a^{-1}x-x|.$

For commuting elements $f,g\in \mathrm{Homeo}_{\mathbb{Z}}(\mathbb{R}),$ we
have that%
\begin{equation*}
l(fg)=\lim_{n\rightarrow \infty }\frac{f^{n}(g^{n}(x))-x}{n}%
=\lim_{n\rightarrow \infty }\frac{f^{n}(g^{n}(x))-g^{n}(x)+g^{n}(x)-x}{n}.
\end{equation*}%
Suppose that $g^{n}(x)=k_{n}+x_{n}$ for $k_{n}\in \mathbb{Z}$ and $x_{n}\in
\lbrack 0,1).$ Then 
\begin{eqnarray*}
&&\lim_{n\rightarrow \infty }\frac{f^{n}(g^{n}x)-g^{n}(x)+g^{n}(x)-x}{n} \\
&=&\lim_{n\rightarrow \infty }\frac{f^{n}(x_{n})-x_{n}+g^{n}(x)-x}{n} \\
&=&\lim_{n\rightarrow \infty }\frac{f^{n}(0)-0+g^{n}(x)-x}{n}=l(f)+l(g).
\end{eqnarray*}%
Therefore, we get $|l(fg)|\leq |l(f)|+|l(g)|.$
\end{proof}

\begin{remark}
It is actually true that the rotation number $l$ is multiplicative on any
amenable group (see \cite{na}, Prop. 2.2.11 and the proof of Prop. 2.2.10,
page 36). This implies that the absolute rotation number $|l|$ is
subadditive on any amenable group. In other words, for any amenable group $G<%
\mathrm{Homeo}_{\mathbb{Z}}(\mathbb{R})$ and any $g,h\in G$ we have $%
|l(gh)|\leq |l(g)|+|l(h)|.$
\end{remark}

\subsection{Asymptotic distortions}

Let $f$ be a $C^{1+bv}$ diffeomorphism of the closed interval $[0,1]$ or the
circle $S^{1}.$ (\textquotedblleft bv\textquotedblright\ means derivative
with finite total variation.) The asymptotic distortion of $f$ is defined as 
\begin{equation*}
l(f)=\mathrm{dist}_{\infty }(f)=\lim_{n\rightarrow \infty }\frac{1}{n}%
\mathrm{var}(\log Df^{n}).
\end{equation*}%
It's proved by Eynard-Bontemps and Navas (\cite{en}, pages 7-8) that

\begin{enumerate}
\item[(1)] $\mathrm{dist}_{\infty }(f^{n})=|n|\mathrm{dist}_{\infty }(f)$
for all $n\in \mathbb{Z}$;

\item[(2)] $\mathrm{dist}_{\infty }(hfh^{-1})=\mathrm{dist}_{\infty }(f)$
for every $C^{1+bv}$ diffeomorphism $h;$

\item[(3)] $\mathrm{dist}_{\infty }(f\circ g)\leq \mathrm{dist}_{\infty }(f)+%
\mathrm{dist}_{\infty }(g)$ for commuting $f,g.$
\end{enumerate}

Therefore, the asymptotic distortion is a length function $l$ on the group $%
\mathrm{Diff}^{1+bv}(M)$ of $C^{1+bv}$ diffeomorphisms for $M=[0,1]$ or $%
S^{1}.$

\subsection{Dynamical degrees of Cremona groups}

Let $k$ be a field and $\mathbb{P}_{k}^{n}=k^{n+1}\backslash \{0\}/\{\lambda
\sim \lambda x:\lambda \neq 0\}$ be the projective space. A rational map
from $\mathbb{P}_{k}^{n}$ to itself is a map of the following type%
\begin{equation*}
(x_{0}:x_{1}:\cdots :x_{n})\dashrightarrow (f_{0}:f_{1}:\cdots :f_{n})
\end{equation*}%
where the $f_{i}$'s are homogeneous polynomials of the same degree without
common factor. The degree of $f$ is $\deg f=\deg f_{i}.$ A birational map
from $\mathbb{P}_{k}^{n}$ to itself is a rational map $f:\mathbb{P}%
_{k}^{n}\dashrightarrow \mathbb{P}_{k}^{n}$ such that there exists a
rational map $g:\mathbb{P}_{k}^{n}\dashrightarrow \mathbb{P}_{k}^{n}$ such
that $f\circ g=g\circ f=\mathrm{id}.$ The group $\mathrm{Bir}(\mathbb{P}%
_{k}^{n})$ of birational maps is called the Cremona group (also denoted as $%
\mathrm{Cr}_{n}(k)$). It is well-known that $\mathrm{Bir}(\mathbb{P}%
_{k}^{n}) $ is isomorphic to the group $\mathrm{Aut}_{k}(k(x_{1},x_{2},%
\ldots ,x_{n}))$ of self-isomorphisms of the field $k(x_{1},x_{2},\ldots
,x_{n})$ of the rational functions in $n$ indeterminates over $k.$ The
(first) dynamical degree $\lambda (f)$ of $f\in \mathrm{Bir}(\mathbb{P}%
_{k}^{n})$ is defined as 
\begin{equation*}
\lambda (f)=\max \{\lim_{n\rightarrow \infty }\deg (f^{n})^{\frac{1}{n}%
},\lim_{n\rightarrow \infty }\deg (f^{-n})^{\frac{1}{n}}\}.
\end{equation*}%
Since $\deg (f^{n})^{\frac{1}{n}}$ is sub-multiplicative, the limit exists.

\begin{lemma}
Let $l(f)=\log \lambda (f).$ Then $l:\mathrm{Bir}(\mathbb{P}%
_{k}^{n})\rightarrow \lbrack 0,+\infty )$ is a length function.
\end{lemma}

\begin{proof}
Without loss of generality, we assume that $\lambda (f)=\lim_{n\rightarrow
\infty }\deg (f^{n})^{\frac{1}{n}},$ while the other case can be considered
similarly. For any $k\in \mathbb{N},$ it is easy that $l(f^{k})=\lim_{n%
\rightarrow \infty }\frac{\log \deg f^{nk}}{n}=kl(f).$ For any $h\in \mathrm{%
Bir}(\mathbb{P}_{k}^{n}),$ we have 
\begin{equation*}
l(hfh^{-1})=\lim_{n\rightarrow \infty }\frac{\log \deg hf^{n}h^{-1}}{n}%
=\lim_{n\rightarrow \infty }\frac{\log \deg f^{nk}}{n}=l(f).
\end{equation*}%
For commuting maps $f,g,$ we have $(fg)^{n}=f^{n}g^{n}.$ Therefore, 
\begin{eqnarray*}
l(fg) &=&\lim_{n\rightarrow \infty }\frac{\log \deg f^{n}g^{n}}{n} \\
&\leq &\lim_{n\rightarrow \infty }\frac{\log \deg f^{n}}{n}%
+\lim_{n\rightarrow \infty }\frac{\log \deg g^{n}}{n}=l(f)+l(g).
\end{eqnarray*}%
This checks the three conditions of the length function.
\end{proof}

It is surprising that when $n=2$ and $k$ is an algebraically closed field,
the length function $l(f)$ is given by the translation length $\tau (f)$ on
an (infinite-dimensional) Gromov $\delta $-hyperbolic space (see
Blanc-Cantat \cite{bc}, Theorem 4.4). Some other length functions are
studied by Blanc and Furter \cite{bd} for groups of birational maps, eg. the
dynamical number of base-points and dynamical length.

\subsection{Margulis invariants}

Let $\mathrm{Aff}(\mathbb{R}^{n})=\mathbb{R}^{n}\rtimes \mathrm{GL}_{n}(%
\mathbb{R})$ be the affine group of the $n$-dimensional space $\mathbb{R}%
^{n} $. For a subgroup $G<\mathrm{Aff}(\mathbb{R}^{n}),$ the linear part is $%
\Phi (G)$, where $\Phi :$ $\mathbb{R}^{n}\rtimes \mathrm{GL}_{n}(\mathbb{R}%
)\rightarrow \mathrm{GL}_{n}(\mathbb{R})$ is the natural quotient map.

\begin{theorem}
\label{margu}Let $G<\mathrm{Aff}(\mathbb{R}^{n}).$ Suppose that

\begin{enumerate}
\item[(i)] the linear part $\Phi (G)$ lies in the isometric group $\mathrm{%
Isom}(B)\ $for a bilinear form $B:\mathbb{R}^{n}\times \mathbb{R}%
^{n}\rightarrow \mathbb{R}.$

\item[(ii)] any $1\neq g\in \Phi (G)$ has $1$ as its eigenvalue and the
corresponding eigenspace is of dimension $1,$ spanned by a unit eigenvector $%
x_{g}\in \mathbb{R}^{n}.$

\item[(iii)] any commuting elements $a,b\in l(G)$ have a common eigenvector $%
0\neq x_{a}=x_{b}\in \mathbb{R}^{n}$ corresponding to the eigenvalue $1.$
\end{enumerate}

Then $l:G\rightarrow \mathbb{R}_{\geq 0}$ given by $l((g,v))=|B(x_{g},v)|$
is a length function.
\end{theorem}

\begin{proof}
Note that the definition of $l$ does not depend on the choice the sign of $%
x_{g}.$ For any integer $n>0,$ we have $(v,g)^{n}=(%
\sum_{i=0}^{n-1}g^{i}v,g^{n})$ for any $(v,g)\in G.$ Therefore,%
\begin{equation*}
B(x_{g},\sum_{i=0}^{n-1}g^{i}v)=nB(x_{g},v),
\end{equation*}%
implying $l((v,g)^{n})=nl(g,v).$ Note that $x_{g}=x_{g^{-1}}$ and 
\begin{eqnarray*}
B(x_{g^{-1}},-g^{-1}v) &=&B(x_{g},-v), \\
l((v,g)) &=&l((v,g)^{-1}).
\end{eqnarray*}%
For any $(u,h)\in G,$ we have 
\begin{equation*}
(u,h)(v,g)(u,h)^{-1}=(u+hv-hgh^{-1}u,hgh^{-1}).
\end{equation*}%
Note that $hgh^{-1}(hx_{g})=hx_{g}$ and $x_{hgh^{-1}}=hx_{g}.$ We have%
\begin{eqnarray*}
l((u,h)(v,g)(u,h)^{-1}) &=&|B(hx_{g},u+hv-hgh^{-1}u)| \\
&=&|B(x_{g},h^{-1}u+v-gh^{-1}u)|=|B(x_{g},v)|.
\end{eqnarray*}%
For any commuting pair $(u,a),(v,b)\in G,$ we have%
\begin{eqnarray*}
l((u,a)(v,b)) &=&l(u+av,ab)=|B(x_{a},u+av)| \\
&=&|B(x_{a},u)+B(x_{a},v)|\leq l((u,a))+l((v,b)).
\end{eqnarray*}
\end{proof}

\begin{corollary}
Let $\alpha _{\phi }:\Gamma \rightarrow \mathbb{R}$ be the Margulis
invariant defined in the section of Introduction. The absolute value
function $|\alpha _{\phi }|$ is a length function.
\end{corollary}

\begin{proof}
By Theorem \ref{margu}, it is enough to check the three conditions. If $\phi
(g)$ is hyperbolic (for an element $g\in \Gamma $), it has three real
eigenvalues with one of them is $1.$ The first two conditions are obvious.
Now we check the third condition. If $g,h$ are commuting elements, the
linear parts $\phi (g),\phi (h)$ lie in a maximal split torus in $\mathrm{SO}%
(2,1)$ and thus conjugate to elements in the standard torus simultaneously.
Let $t$ be a nonzero real number. The adjoint representation of $\mathrm{SL}%
(2,\mathbb{R})$ on the Lie algebra $\mathfrak{sl}(2,\mathbb{R})$ gives an
embedding $f:\mathrm{PSL}_{2}(\mathbb{R})\rightarrow \mathrm{SO}(2,1).$ Note
that%
\begin{equation*}
\begin{bmatrix}
t & 0 \\ 
0 & 1/t%
\end{bmatrix}%
\begin{bmatrix}
a & b \\ 
c & -a%
\end{bmatrix}%
\begin{bmatrix}
1/t & 0 \\ 
0 & t%
\end{bmatrix}%
=%
\begin{bmatrix}
a & t^{2}b \\ 
c/t^{2} & -a%
\end{bmatrix}%
.
\end{equation*}%
Therefore, 
\begin{equation*}
\begin{bmatrix}
1 & 0 \\ 
0 & -1%
\end{bmatrix}%
\in \mathfrak{sl}(2,\mathbb{R})
\end{equation*}%
is a common eigenvector for elements in the torus 
\begin{equation*}
\{f(%
\begin{bmatrix}
t & 0 \\ 
0 & 1/t%
\end{bmatrix}%
):t\in \mathbb{R}^{\ast }\}.
\end{equation*}%
This checked condition 3). The claim follows from Theorem \ref{margu}.
\end{proof}

The Margulis invariant has essential applications in the study of properly
discontinuously affine actions on the Euclidean space $\mathbb{R}^{n}.$ For
more information, see the survey article \cite{Ab}.

\section{Groups with purely positive length functions}

\begin{definition}
A length function $l$ on a group $G$ is said to be purely positive if $%
l(g)>0 $ for any infinite-order element $g\in G.$
\end{definition}

In this section, we show that the (Gromov) hyperbolic group, mapping class
group and outer automorphism groups of free groups have purely positive
length functions. First, let us recall the relevant definitions.

A geodesic metric space $X$ is $\delta $-hyperbolic (for some real number $%
\delta >0$) if for any geodesic triangle $\Delta xyz$ in $X,$ one side is
contained the $\delta $-neighborhood of the other two sides. A group $G$ is
(Gromov) hyperbolic if $G$ acts properly discontinuously and cocompactly on
a $\delta $-hyperbolic space $X$.

\begin{definition}

\begin{enumerate}
\item[(i)] An element $g$ in a group $G$ is called primitive if it cannot be
{}written as a proper power $\alpha ^{n},$ where $\alpha \in G$ and $|n|\geq
2;$

\item[(ii)] A group $G$ has unique-root property if every infinite-order
element $g$ is a power of a unique (up to sign) primitive element, i.e. $%
g=\gamma ^{n}=\gamma _{1}^{m}$ for primitive elements $\gamma ,\gamma _{1}$
will imply $\gamma =\gamma _{1}^{\pm }.$
\end{enumerate}
\end{definition}

The following fact is well-known.

\begin{lemma}
A torsion-free hyperbolic group has unique-root property.
\end{lemma}

\begin{proof}
Let $G$ be a torsion-free hyperbolic group and $1\neq g\in G.$ Suppose that $%
g=\gamma ^{n}=\gamma _{1}^{m}$ for primitive elements $\gamma $ and $\gamma
_{1}.$ The set $C_{G}(g)$ of centralizers is virtually cyclic (cf. \cite{bh}%
, Corollary 3.10, page 462). By a result of Serre, a torsion-free virtually
free group is free. Since $G$ is torsion-free, the group $C_{G}(g)$ is thus
free and thus cyclic, say generated by $t$. Since $\gamma $ and $\gamma _{1}$
are primitive, they are $t^{\pm }.$
\end{proof}

\begin{remark}
The previous lemma does not hold for general hyperbolic groups with
torsions. For example, let $G=\mathbb{Z}/2\times \mathbb{Z}.$ We have $%
(0,2)=(0,1)^{2}=(1,1)^{2}$ and $(0,1),$ $(1,1)$ are both primitive.
\end{remark}

For a group $G,$ let $P(G)$ be the set of all primitive elements. We call
two primitive elements $\gamma ,\gamma ^{\prime }$ are generally conjugate
if there exists $g\in G$ such that $g\gamma g^{-1}=\gamma ^{\prime }$ or $%
g\gamma ^{-1}g=\gamma ^{\prime }$. Let $\mathrm{CP}(G)$ be the general
conjugacy classes of primitive elements. For a set $S,$ let $S_{\mathbb{R}}$
be the set of all real functions on $S$. The convex polyhedral cone spanned
by $S$ is the subset 
\begin{equation*}
\{\Sigma _{s\in S}a_{s}s\mid a_{s}\geq 0\ \}\subset S_{\mathbb{R}}.
\end{equation*}

\begin{lemma}
Let $G$ be a torsion-free hyperbolic group. The set of all length functions
on $G$ is the convex polyhedral cone spanned by the general conjugacy
classes $\mathrm{CP}(G).$
\end{lemma}

\begin{proof}
Let $l$ be a length function on $G.$ Then $l$ gives an element $\Sigma
_{s\in S}a_{s}s$ in the convex polyhedral cone by $a_{s}=l(s).$ Conversely,
for any general conjugacy classes $[s]\in \mathrm{CP}(G)$ with $s$ a
primitive element, let $l_{s}$ be the function defined by $l_{s}(s^{\pm })=1$
and $l_{s}(\gamma )=0$ for element $\gamma $ in any other general conjugacy
classes. For any $1\neq g\in G,$ there is a unique (up to sign) primitive
element $\gamma $ such that $g=\gamma ^{n}.$ Define $l_{s}(g)=|n|l_{s}(%
\gamma ).$ Then $l_{s}$ satisfies conditions (1) and (2) in Definition \ref%
{def}. Condition (3) is satisfied automatically, since any commuting pair of
elements $a,b$ generates a cyclic group in a torsion-free hyperbolic group.
Any element $\Sigma _{s\in S}a_{s}s$ gives a length function on $G$ as a
combination of $a_{s}l_{s}$.
\end{proof}

\begin{lemma}
\label{pol}Let $G\ $be one of the following groups:

\begin{itemize}
\item automorphism group $\mathrm{Aut}(F_{k})$ of a free group$;$

\item outer automorphism group $\mathrm{Out}(F_{k})$ of a free group or

\item mapping class group $\mathrm{MCG}(\Sigma _{g,m})$ (where $\Sigma
_{g,m} $ is an oriented surface of genus $g$ and $m$ punctures);

\item a hyperbolic group,

\item a $\mathrm{CAT}(0)$ group or more generally

\item a semi-hyperbolic group,

\item a group acting properly semi-simply on a $\mathrm{CAT}(0)$ space,

\item a group acting properly semi-simply on a $\delta $-hyperbolic space.
\end{itemize}

Then $G$ has a purely positive length function.
\end{lemma}

\begin{proof}
Note that hyperbolic groups and $\mathrm{CAT}(0)$ groups are semihyperbolic
(see \cite{bh}, Prop. 4.6 and Cor. 4.8, Chapter III.$\Gamma $). For a
semihyperbolic group $G$ acting a metric space $X$ (actually $X=G$), the
translation $\tau $ is a length function by Lemma \ref{trans}. Moreover, for
any infinite-order element $g\in G,$ the length $\tau (g)>0$ (cf. \cite{bh},
Lemma 4.18, page 479). For group acting properly semisimply on a $\mathrm{CAT%
}(0)$ space (or a $\delta $-hyperbolic space), the translation 
\begin{equation*}
l(\gamma )=\lim_{n\rightarrow \infty }\frac{d(x,\gamma ^{n}x)}{n}
\end{equation*}%
is a length function (cf. Lemma \ref{trans}). For any hyperbolic $\gamma $,
we get $l(\gamma )>0.$ For any elliptic $\gamma ,$ it is finite-order since
the action is proper.

Alibegovic \cite{al} proves that the stable word length of $\mathrm{Aut}%
(F_{n}),\mathrm{Out}(F_{n})$ are purely positive. Farb, Lubotzky and Minsky 
\cite{flm} prove that Dehn twists and more generally all elements of
infinite order in $\mathrm{MCG}(\Sigma _{g,m})$ have positive translation
length.
\end{proof}

\bigskip

\begin{definition}
A group $G$ is called poly-positive (or has a poly-positive length), if
there is a subnormal series%
\begin{equation*}
1=H_{n}\vartriangleleft H_{n-1}\vartriangleleft \cdots \vartriangleleft
H_{0}=G
\end{equation*}%
such that every finitely generated subgroup of the quotient $H_{i}/H_{i+1}$ $%
(i=0,...,n-1)$ has a purely positive length function.
\end{definition}

Recall that a group $G$ is poly-free, if there is a subnormal series $%
1=H_{n}\vartriangleleft H_{n-1}\vartriangleleft \cdots \vartriangleleft
H_{0}=G$ such that the successive quotient $H_{i}/H_{i+1}$ is free $%
(i=0,...,n-1).$ Since a free group is hyperbolic, it has a purely positive
length function. This implies that a poly-free group is poly-positive. {}A
group is said to have a virtual property if a finite-index subgroup has the
property.

Let $\Sigma $ be a closed oriented surface endowed with an area form $\omega
.$ Denote by $\mathrm{Diff}(\Sigma ,\omega )$ the group of diffeomorphisms
preserving $\omega $ and $\mathrm{Diff}_{0}(\Sigma ,\omega )$ the subgroup
consisting of diffeomorphisms isotopic to the identity.

\begin{lemma}
\label{area-pres}When the genus of $\Sigma $ is greater than $1,$ the group $%
\mathrm{Diff}_{0}(\Sigma ,\omega )$ and $\mathrm{Diff}(\Sigma ,\omega )$ are
poly-positive.
\end{lemma}

\begin{proof}
This is {}essentially proved by Py \cite{py} (Section 1). There is a group
homomorphism%
\begin{equation*}
\alpha :\mathrm{Diff}_{0}(\Sigma ,\omega )\rightarrow H_{1}(\Sigma ,\mathbb{R%
})
\end{equation*}%
with $\ker \alpha =\mathrm{Ham}(\Sigma ,\omega )$ the group of Hamiltonian
diffeomorphisms of $\Sigma .$ Polterovich \cite{po} (1.6.C.) proves that any
finitely generated group of $\mathrm{Ham}(\Sigma ,\omega )$ has a purely
positive stable word length. Since the quotient group $\mathrm{Diff}(\Sigma
,\omega )/\mathrm{Diff}_{0}(\Sigma ,\omega )$ is a subgroup of the mapping
class group $\mathrm{MCG}(\Sigma ),$ which has a purely positive stable word
length by Farb-Lubotzky-Minsky \cite{flm}, the group $\mathrm{Diff}(\Sigma
,\omega )$ is poly-positive.
\end{proof}

\section{Vanishing of length functions on abelian-by-cyclic groups\label%
{section4}}

We will need the following result proved in \cite{tao}.

\begin{lemma}
\label{tao}Given a group $G$, let $l:G\rightarrow \lbrack 0,+\infty )$ be
function such that

1) $l(e)=0;$

2) $l(x^{n})=|n|l(x)$ for any $x\in G,$ any $n\in \mathbb{Z};$

3) $l(xy)\leq l(x)+l(y)$ for any $x,y\in G.$

Then there exist a real Banach space $(\mathbb{B},\Vert \Vert )$ and a group
homomorphism $\varphi :G\rightarrow \mathbb{B}$ such that $l(x)=\Vert
\varphi (x)\Vert $ for all $x\in G.$ Furthermore, if $l(x)>0$ for any $x\neq
e$, one can take $\varphi $ to be injective, i.e., an isometric embedding.
\end{lemma}

Let $\mathbb{Z}^{2}\rtimes _{A}\mathbb{Z}$ be an abelian-by-cyclic group,
where $A=%
\begin{bmatrix}
a & b \\ 
c & d%
\end{bmatrix}%
\in \mathrm{GL}_{2}(\mathbb{Z})$. We prove Theorem \ref{theorem0} by proving
the following two theorems.

\begin{theorem}
When the absolute value of the trace $|\mathrm{tr}(A)|>2,$ any length
function $l:\mathbb{Z}^{2}\rtimes _{A}\mathbb{Z}\rightarrow \mathbb{R}_{\geq
0}$ vanishes on $\mathbb{Z}^{2}.$
\end{theorem}

\begin{proof}
Let $A=%
\begin{bmatrix}
a & b \\ 
c & d%
\end{bmatrix}%
\in \mathrm{GL}_{2}(\mathbb{Z}).$ Suppose that $t$ is a generator of $%
\mathbb{Z}$ and%
\begin{equation*}
t%
\begin{bmatrix}
x \\ 
y%
\end{bmatrix}%
t^{-1}=A%
\begin{bmatrix}
x \\ 
y%
\end{bmatrix}%
\end{equation*}%
for any $x,y\in \mathbb{Z}.$ Note that 
\begin{equation*}
(0,t^{k})(v,0)(0,t^{k})^{-1}=(A^{k}v,0)
\end{equation*}%
for any $v\in \mathbb{Z}^{2}$ and $k\in \mathbb{Z}$. Therefore, an element $%
v\in \mathbb{Z}^{2}$ is conjugate to $A^{k}v$ for any integer $k.$ Note that 
\begin{equation*}
A%
\begin{bmatrix}
1 \\ 
0%
\end{bmatrix}%
=%
\begin{bmatrix}
a \\ 
c%
\end{bmatrix}%
,A^{2}%
\begin{bmatrix}
1 \\ 
0%
\end{bmatrix}%
=%
\begin{bmatrix}
a^{2}+bc \\ 
ac+dc%
\end{bmatrix}%
\end{equation*}%
and 
\begin{equation*}
\begin{bmatrix}
a^{2}+bc \\ 
ac+dc%
\end{bmatrix}%
=(a+d)%
\begin{bmatrix}
a \\ 
c%
\end{bmatrix}%
-(ad-bc)%
\begin{bmatrix}
1 \\ 
0%
\end{bmatrix}%
.
\end{equation*}%
Therefore, we have%
\begin{eqnarray*}
|a+d|l(%
\begin{bmatrix}
1 \\ 
0%
\end{bmatrix}%
) &=&l((a+d)%
\begin{bmatrix}
a \\ 
c%
\end{bmatrix}%
) \\
&=&l(%
\begin{bmatrix}
a^{2}+bc \\ 
ac+dc%
\end{bmatrix}%
+(ad-bc)%
\begin{bmatrix}
1 \\ 
0%
\end{bmatrix}%
) \\
&\leq &(1+|ad-bc|)l(%
\begin{bmatrix}
1 \\ 
0%
\end{bmatrix}%
).
\end{eqnarray*}%
When $ad-bc=\pm 1$ and $|a+d|>2,$ we must have $l(%
\begin{bmatrix}
1 \\ 
0%
\end{bmatrix}%
)=0.$ Similarly, we can prove that $l(%
\begin{bmatrix}
0 \\ 
1%
\end{bmatrix}%
)=0.$ Since $l$ is subadditive on $\mathbb{Z}^{2},$ we get that $l$ vanishes
on $\mathbb{Z}^{2}.$
\end{proof}

\begin{theorem}
When the absolute value $|\mathrm{tr}(A)|=|a+d|=2,I_{2}\neq A\in \mathrm{SL}%
_{2}(\mathbb{Z}),$ any length function $l:\mathbb{Z}^{2}\rtimes _{A}\mathbb{Z%
}\rightarrow \mathbb{R}_{\geq 0}$ vanishes on the direct summand of $\mathbb{%
Z}^{2}$ spanned by eigenvectors of $A$.
\end{theorem}

\begin{proof}
We may assume that $A=%
\begin{bmatrix}
1 & n \\ 
0 & 1%
\end{bmatrix}%
,n\neq 0.$ For any integer $k\geq 0$ and $v\in \mathbb{Z}^{2},$ we have%
\begin{equation*}
t^{k}vt^{-k}=A^{k}v.
\end{equation*}%
Take $v=%
\begin{bmatrix}
0 \\ 
1%
\end{bmatrix}%
$ to get that 
\begin{equation*}
t^{k}%
\begin{bmatrix}
0 \\ 
1%
\end{bmatrix}%
t^{-k}=%
\begin{bmatrix}
kn \\ 
0%
\end{bmatrix}%
+%
\begin{bmatrix}
0 \\ 
1%
\end{bmatrix}%
.
\end{equation*}%
Since the function $l|_{\mathbb{Z}^{2}}$ is given by the norm of a Banach
space according to Lemma \ref{tao}, we get that%
\begin{eqnarray*}
k|n|l(%
\begin{bmatrix}
1 \\ 
0%
\end{bmatrix}%
) &\leq &l(t^{k}%
\begin{bmatrix}
0 \\ 
1%
\end{bmatrix}%
t^{-k})+l(%
\begin{bmatrix}
0 \\ 
1%
\end{bmatrix}%
) \\
&=&2l(%
\begin{bmatrix}
0 \\ 
1%
\end{bmatrix}%
).
\end{eqnarray*}%
Since $k$ is arbitrary, we get that $l(%
\begin{bmatrix}
1 \\ 
0%
\end{bmatrix}%
)=0.$
\end{proof}

\begin{remark}
When $A=%
\begin{bmatrix}
1 & 1 \\ 
0 & 1%
\end{bmatrix}%
,$ the semidirect product $G=\mathbb{Z}^{2}\rtimes _{A}\mathbb{Z}$ is a
Heisenberg group. A length function on $G/Z(G)\cong \mathbb{Z}^{2}$ gives a
length function on $G.$ In particular, a length function of $G$ may not
vanish on the second component $%
\begin{bmatrix}
0 \\ 
1%
\end{bmatrix}%
\in \mathbb{Z}^{2}<G.$
\end{remark}

\begin{remark}
When $A\in \mathrm{SL}_{2}(\mathbb{Z})$ has $|\mathrm{tr}(A)|<2,$ the matrix 
$A$ is of finite order and the semi-direct product $\mathbb{Z}^{2}\rtimes
_{A}\mathbb{Z}$ contains $\mathbb{Z}^{3}$ as a finite-index normal subgroup.
Actually, in this case the group $\mathbb{Z}^{2}\rtimes _{A}\mathbb{Z}$ is
the fundamental group of a flat $3$-manifold $M$ (see \cite{wol}, Theorem
3.5.5). Therefore, the group $\mathbb{Z}^{2}\rtimes _{A}\mathbb{Z}$ acts
freely properly discontinuously isometrically and cocompactly on the
universal cover $\tilde{M}=\mathbb{R}^{3}.$ This means the translation
length gives a purely positive length function on $\mathbb{Z}^{2}\rtimes _{A}%
\mathbb{Z}$.
\end{remark}

\begin{lemma}
Let $A\in \mathrm{GL}_{n}(\mathbb{Z})$ be a matrix and $G=\mathbb{Z}%
^{n}\rtimes _{A}\mathbb{Z}$ the semi-direct product. Let $%
\sum_{i=0}^{n}a_{i}x^{i}$ be the {}characteristic polynomial of some power $%
A^{k}$. Suppose that for some $k,$ there is a coefficient $a_{i}$ such that 
\begin{equation*}
|a_{i}|>\sum_{j\neq i}|a_{j}|.
\end{equation*}%
Any length function $l$ of $G$ vanishes on $\mathbb{Z}^{n}.$
\end{lemma}

\begin{proof}
Let $t$ be a generator of $Z$ and $tat^{-1}=Aa$ for any $a\in \mathbb{Z}%
^{n}. $ Note that for any integer $m,$ we have $t^{m}at^{-m}=A^{m}a$ and $%
l(a)=l(A^{m}a).$ Note that 
\begin{equation*}
\sum_{i=0}^{n}a_{i}A^{ki}=0
\end{equation*}%
and thus%
\begin{eqnarray*}
\sum_{i=0}^{n}a_{i}A^{ki}a &=&0, \\
l(\sum_{i=0}^{n}a_{i}A^{ki}a) &=&0
\end{eqnarray*}%
for any $a\in \mathbb{Z}^{n}.$ Therefore,%
\begin{equation*}
|a_{i}|l(a)=|a_{i}|l(A^{ki}a)=l(\sum_{j\neq i}a_{j}A^{kj}a)\leq \sum_{j\neq
i}|a_{j}|l(a).
\end{equation*}%
This implies that $l(a)=0.$
\end{proof}

\begin{corollary}
Let $G$ be a finitely generated nilpotent group, which is not virtually
abelian. Any length function $l:G\rightarrow \mathbb{R}_{\geq 0}$ vanishes
at an infinite-order element $g.$
\end{corollary}

\begin{proof}
The group $G$ contains a finite-index torsion-free nilpotent subgroup $H.$
Suppose that $H$ is not abelian. Denote by $H_{0}=H$ and $H_{i}=[H_{i-1},H]$
for each positive integer $i.$ Since $H$ is nilpotent, there is an integer $%
i=n\geq 1$ such that $H_{n}\neq 1$ is abelian. Choose $a,b\in H_{n-1}$ such
that $[a,b]=:c\neq 1\in H_{n}.$ Since $H_{n+1}$ is trivial, the elements $%
a,b $ commute with $c.$ This implies that the subgroup generated by $a,b,c$
is a Heisenberg group. Theorem \ref{theorem0} implies that $l(c)=0.$
\end{proof}

\begin{remark}
The idea of the proofs of Theorem \ref{theorem0} is already known for the
word metric or the stable norms (for example, see \cite{lmr1}\cite{alperin}%
\cite{conner}). Actually, Connor \cite{conner} obtained a general vanishing
result for the stable norms on solvable groups. An analog will be of great
interest for the vanishing of entropies and Lyapunov exponents (as in
Corollary \ref{corllary0}). Currently, we don't know whether such an analog
holds for general length functions.
\end{remark}

\begin{proof}[Proof of Corollary \protect\ref{corllary0}]
When the group action is $C^{2},$ define 
\begin{equation*}
l(f)=\max \{\lim_{n\rightarrow +\infty }\frac{\log \sup_{x\in M}\Vert
D_{x}f^{n}\Vert }{n},\lim_{n\rightarrow +\infty }\frac{\log \sup_{x\in
M}\Vert D_{x}f^{-n}\Vert }{n}\}
\end{equation*}%
for any diffeomorphism $f:M\rightarrow M.$ Lemma \ref{derivelength} shows
that $l$ is a length function, which is an upper bound of the Lyapunov
exponents. When the group action is Lipschitz, define 
\begin{equation*}
L(f)=\sup_{x\neq y}\frac{d(fx,fy)}{d(x,y)}
\end{equation*}%
for a Lipschitz-homeomorphism $f:M\rightarrow M.$ Since $L(fg)\leq L(f)L(g)$
for two Lipschitz-homeomorphisms $f,g:M\rightarrow M,$ we have that $%
l(f):=\lim_{n\rightarrow \infty }\frac{\max \{\log (L(f^{n})),\log
(L(f^{-n}))\}}{n}$ gives a length function by Lemma \ref{lem0}. Note that $%
l(f)\geq h_{top}(f)$ (see \cite{kat}, Theorem 3.2.9, page 124). The
vanishings of the topological entropy $h_{top}$ and the Lyapunov exponents
in Corollary \ref{corllary0} are proved by Theorem \ref{theorem0}
considering these length functions.
\end{proof}

\section{Classification of length functions on nilpotent groups}

The following lemma is a key step for our proof of the vanishing of length
functions on Heisenberg groups.

\begin{lemma}
\label{conj}Let 
\begin{equation*}
G=\langle a,b,c\mid aba^{-1}b^{-1}=c,ac=ca,bc=cb\rangle
\end{equation*}
be the Heisenberg group. Suppose that $f:G\rightarrow \mathbb{R}$ is a
conjugation-invariant function, i.e. $f(xgx^{-1})=f(g)$ for any $x,g\in G.$
For any coprime integers (not-all-zero) $m,n$ and any integer $k,$ we have 
\begin{equation*}
f(a^{m}b^{n}c^{k})=f(a^{m}b^{n}).
\end{equation*}
\end{lemma}

\begin{proof}
It is well-known that for any integers $n,m,$ we have $[a^{n},b^{m}]=c^{nm}.$
Actually, since $aba^{-1}b^{-1}=c,$ we have $ba^{-1}b^{-1}=a^{-1}c$ and thus 
$ba^{-n}b^{-1}=a^{-n}c^{n}$ for any integer $n.$ Therefore, $%
a^{n}ba^{-n}b^{-1}=c^{n}$ and $a^{n}ba^{-n}=c^{n}b,$ $%
a^{n}b^{m}a^{-n}=c^{nm}b^{m}$ for any integer $m.$ This means $%
[a^{n},b^{m}]=c^{nm}.$ For any coprime $m,n$, and any integer $k,$ let $%
s,t\in \mathbb{Z}$ such that $ms+nt=k.$ We have 
\begin{equation*}
a^{-m}b^{-s}a^{m}b^{s}=c^{ms},b^{-s}a^{m}b^{s}=a^{m}c^{ms},b^{-s}a^{m}b^{n}b^{s}=a^{m}b^{n}c^{ms}
\end{equation*}%
and 
\begin{equation*}
a^{t}b^{n}a^{-t}b^{-n}=c^{nt},a^{t}b^{n}a^{-t}=b^{n}c^{nt},a^{t}a^{m}b^{n}a^{-t}=a^{m}b^{n}c^{nt}.
\end{equation*}%
Therefore,%
\begin{equation*}
a^{t}(b^{-s}a^{m}b^{n}b^{s})a^{-t}=a^{t}(a^{m}b^{n}c^{ms})a^{-t}=a^{m}b^{n}c^{nt+ms}=a^{m}b^{n}c^{k}.
\end{equation*}%
When $f$ is conjugation-invariant, we get $f(a^{m}b^{n}c^{k})=f(a^{m}b^{n})$
for any coprime $m,n$, and any integer $k.$
\end{proof}

\begin{lemma}
\label{nilpfac}Let 
\begin{equation*}
G=\langle a,b,c\mid aba^{-1}b^{-1}=c,ac=ca,bc=cb\rangle
\end{equation*}%
be the Heisenberg group. Any length function $l:G\rightarrow \lbrack
0,\infty )$ (in the sense of Definition \ref{def}) factors through the
abelianization $G_{\mathrm{ab}}:=G/[G,G]\cong \mathbb{Z}^{2}.$ In other
words, there is a function 
\begin{equation*}
l^{\prime }:G_{\mathrm{ab}}\rightarrow \lbrack 0,\infty )
\end{equation*}%
such that $l^{\prime }(x^{n})=|n|l^{\prime }(x)$ for any $x\in G_{\mathrm{ab}%
},$ any integer $n$ and $l=l^{\prime }\circ q$, where $q:G\rightarrow G_{%
\mathrm{ab}}$ is the natural quotient group homomorphism.
\end{lemma}

\begin{proof}
Let $H=\langle c\rangle \cong \mathbb{Z}$ and write $G=\cup gH$ the union of
left cosets. We choose the representative $g_{ij}=a^{i}b^{j}$ with $(i,j)\in 
\mathbb{Z}^{2}.$ Note that the subgroup $\langle g_{ij},c\rangle $ generated
by $g_{ij},c$ is isomorphic to $\mathbb{Z}^{2}$ for coprime $i,j.$ The
length function $l$ is subadditive on $\langle g_{ij},c\rangle .$ By Lemma %
\ref{tao}, there is a Banach space $\mathbb{B}$ and a group homomorphism $%
\varphi :$ $\langle g_{ij},c\rangle \rightarrow \mathbb{B}$ such that $%
l(g)=\Vert \varphi (g)\Vert $ for any $g\in $ $\langle g_{ij},c\rangle .$
Lemma \ref{conj} implies that 
\begin{equation*}
\Vert \varphi (g_{ij})+\varphi (c^{k})\Vert =\Vert \varphi (g_{ij})\Vert
\end{equation*}%
for any integer $k.$ Since 
\begin{equation*}
\Vert \varphi (g_{ij})\Vert =\Vert \varphi (g_{ij})+\varphi (c^{k})\Vert
\geq |k|\Vert \varphi (c)\Vert -\Vert \varphi (g_{ij})\Vert
\end{equation*}%
for any $k,$ we have that $\Vert \varphi (c)\Vert =l(c)=0.$ This implies
that 
\begin{equation*}
l(g_{ij}^{n}c^{m})=\Vert \varphi (g_{ij}^{n})+\varphi (c^{m})\Vert
=l(g_{ij}^{n})
\end{equation*}%
for any integers $m,n.$ Moreover, for any integers $m,n$ and coprime $i,j,$
we have $a^{ni}b^{nj}c^{m}=(a^{i}b^{j})^{n}c^{k}$ for some integer $k.$ This
implies that 
\begin{equation*}
l(a^{ni}b^{nj}c^{m})=l((a^{i}b^{j})^{n})=|n|l(a^{i}b^{j}).
\end{equation*}%
Therefore, the function $l$ is constant on each coset $gH.$ Define $%
l^{\prime }(gH)=l(g).$ Since $l(g^{k})=|k|l(g),$ we have that $l^{\prime
}(g^{k}H)=|k|l^{\prime }(gH).$ The proof is finished.
\end{proof}

Denote by $S^{\prime }=\{(m,n)\mid m,n$ are coprime integers, $(m,n)\neq
(0,0)\}$ be the set of coprime integer pairs and define an equivalence
relation by $(m,n)\sim (m^{\prime },n^{\prime })$ if $(m,n)=\pm (m^{\prime
},n^{\prime })$. Let $S=S^{\prime }/\sim $ be the equivalence classes.

\begin{theorem}
\label{th00}Let $G=\langle a,b,c\mid aba^{-1}b^{-1}=c,ac=ca,bc=cb\rangle $
be the Heisenberg group. The set of all length functions $l:G\rightarrow
\lbrack 0,\infty )$ (in the sense of Definition \ref{def}) is the convex
polyhedral cone 
\begin{equation*}
\mathbb{R}_{\geq 0}[S]=\{\sum_{s\in S}a_{s}s\mid a_{s}\in \mathbb{R}_{\geq
0},s\in S\}.
\end{equation*}
\end{theorem}

\begin{proof}
Similar to the proof of the previous lemma, we let $H=\langle c\rangle \cong 
\mathbb{Z}$ and write $G=\cup gH$ the union of left cosets. We choose the
representative $g_{ij}=a^{i}b^{j}$ with $(i,j)\in \mathbb{Z}^{2}.$ Let $T$
be the set of length functions $l:G\rightarrow \lbrack 0,\infty ).$ For any
length function $l,$ let $\varphi (l)=\sum_{s\in S}a_{s}s\in \mathbb{R}%
_{\geq 0}[S],$ where $a_{s}=l(a^{i}b^{j})$ with $(i,j)$ a representative of $%
s.$ Note that 
\begin{equation*}
l(a^{-i}b^{-j})=l(b^{-j}a^{-i}c^{ij})=l(b^{-j}a^{-i})=l(a^{i}b^{j}),
\end{equation*}%
which implies that $a_{s}$ is well-defined. We have defined a function $%
\varphi :T\rightarrow \mathbb{R}_{\geq 0}[S].$ If $\varphi (l_{1})=\varphi
(l_{2})$ for two functions $l_{1},l_{2},$ then $%
l_{1}(a^{i}b^{j})=l_{2}(a^{i}b^{j})$ for coprime integers $i,j.$ Since both $%
l_{1},l_{2}$ are conjugation-invariant, Lemma \ref{conj} implies that $%
l_{1},l_{2}$ coincide on any coset $a^{i}b^{j}H$ and thus on the whole group 
$G.$ This proves the injectivity of $\varphi .$ For any $\sum_{s\in
S}a_{s}s, $ we define a function 
\begin{equation*}
l:G=\cup a^{i}b^{j}H\rightarrow \lbrack 0,\infty ).
\end{equation*}%
For any coprime integers $i,j,$ define $l(a^{i}b^{j}z)=a_{s}$ for any
representative $(i,j)$ of $s$ and any $z\in H.$ For any general integers $%
m,n $ and $z\in H,$ define $l(a^{m}b^{n}z)=l(a^{m}b^{n})=|\gcd
(m,n)|l(a^{m/\gcd (m,n)}b^{n/\gcd (m,n)})$ and $l(z)=0.$ From the
definition, it is obvious that $l$ is {}homogeneous. Note that any element
of $G$ is of the form $a^{k}b^{s}c^{t}$ for integers $k,s,t\in \mathbb{Z}.$
For any two elements $a^{k}b^{s}c^{t},a^{k^{\prime }}b^{s^{\prime
}}c^{t^{\prime }}$ we have the conjugation%
\begin{equation*}
a^{k^{\prime }}b^{s^{\prime }}c^{t^{\prime }}a^{k}b^{s}c^{t}(a^{k^{\prime
}}b^{s^{\prime }}c^{t^{\prime }})^{-1}=a^{k}b^{s}c^{t^{\prime \prime }}
\end{equation*}%
for some $t^{\prime \prime }\in \mathbb{Z}.$ Therefore, we see that $l$ is
conjugation-invariant. The previous equality also shows that two elements $%
g=a^{k}b^{s}c^{t},h=a^{k^{\prime }}b^{s^{\prime }}c^{t^{\prime }}$ are
commuting if and only if $ks^{\prime }=k^{\prime }s$, which is also
equivalent to that they lies {}simultaneously in $\langle
a^{i}b^{j},c\rangle $ for a pair of coprime integers $i,j.$ By construction,
we have $l(gh)=l(g)+l(h).$ This proves the surjectivity of $\varphi .$
\end{proof}

\section{Length functions on matrix groups\label{section6}}

In this section, we study length functions on matrix groups $\mathrm{SL}_{n}(%
\mathbb{R})$. As the proofs are elementary, we present them here in a
separate section, without using profound results on Lie groups and algebraic
groups. The following lemma is obvious.

\begin{lemma}
\label{bs}Let $G_{p,q}=\langle x,t:tx^{p}t^{-1}=x^{q}\rangle $ be a
Baumslag-Solitar group. When $|p|\neq |q|,$ any length function $l$ on $G$
has $l(x)=0.$
\end{lemma}

\begin{proof}
Note that $|p|l(x)=l(x^{p})=l(x^{q})=|q|l(x),$ which implies $l(x)=0.$
\end{proof}

Let $V^{n}$ be a finite-dimensional vector space over a field $K$ and $%
A:V\rightarrow V$ a nilpotent linear transformation (i.e. $A^{k}=0$ for some
positive integer $k).$ The following fact is from linear algebra (see the
Lemma of page 313 in \cite{beida}. Since the reference is in Chinese, we
repeat the proof here).

\begin{lemma}
\label{jordan} $I+A$ is conjugate to a direct sum of Jordan blocks with $1s$
along the diagonal.
\end{lemma}

\begin{proof}
We prove that $V$ has a basis 
\begin{equation*}
\{a_{1},Aa_{1},\ldots ,A^{k_{1}-1}a_{1},a_{2},Aa_{2},\ldots
,A^{k_{2}-1}a_{2},\ldots ,a_{s},\ldots ,Aa_{s},\ldots ,A^{k_{s}-1}a_{s}\}
\end{equation*}%
satisfying $A^{k_{i}}a_{i}=0$ for each $i,$ which implies that the
representation matrix of $I+A$ is similar to a direct sum of Jordan blocks
with $1$ along the diagonal. The proof is based on the induction of $\dim V.$
When $\dim V=1,$ choose $0\neq v\in V.$ Suppose that $Av=\lambda v.$ Then $%
A^{k}v=\lambda ^{k}v=0$ and thus $\lambda =0.$ Suppose that the case is
proved for vector spaces of dimension $k<n.$ Note that the invariant
subspace $AV$ is a proper subspace of $V$ (otherwise, $AV=V$ implies $%
A^{k}V=A^{k-1}V=V=0$). By induction, the subspace $AV$ has a basis 
\begin{equation*}
\{a_{1},Aa_{1},\ldots ,A^{k_{1}-1}a_{1},a_{2},Aa_{2},\ldots
,A^{k_{2}-1}a_{2},\ldots ,a_{s},\ldots ,Aa_{s},\ldots ,A^{k_{s}-1}a_{s}\}.
\end{equation*}%
Choose $b_{i}\in V$ satisfying $A(b_{i})=a_{i}.$ Then $A$ maps the set 
\begin{equation*}
\{b_{1},Ab_{1},\ldots ,A^{k_{1}}b_{1},b_{2},Ab_{2},\ldots
,A^{k_{2}}b_{2},b_{s},\ldots ,Ab_{s},\ldots ,A^{k_{s}}b_{s}\}
\end{equation*}%
to the basis 
\begin{equation*}
\{a_{1},Aa_{1},\ldots ,A^{k_{1}-1}a_{1},a_{2},Aa_{2},\ldots
,A^{k_{2}-1}a_{2},\ldots ,a_{s},\ldots ,Aa_{s},\ldots ,A^{k_{s}-1}a_{s}\}.
\end{equation*}%
This implies that the former set is linearly independent (noting that $%
A(A^{k_{i}}b_{i})=0$). Extend this set to be a $V^{\prime }s$ basis 
\begin{equation*}
\{b_{1},Ab_{1},\ldots ,A^{k_{1}}b_{1},b_{2},Ab_{2},\ldots
,A^{k_{2}}b_{2},b_{s},\ldots ,Ab_{s},\ldots ,A^{k_{s}}b_{s},b_{s+1},\ldots
,b_{s^{\prime }}\}.
\end{equation*}%
Note that $Ab_{i}=0$ for $i\geq s+1$ and $A^{k_{i}+1}b_{i}=A^{k_{i}}a_{i}=0$
for each $i\leq s.$ This finishes the proof.
\end{proof}

\begin{corollary}
\label{unip}Let $A_{n\times n}$ be a strictly upper triangular matrix over a
field $K$ of characteristic $\mathrm{ch}(K)\neq 2$. Then $A^{2}$ is
conjugate to $A.$
\end{corollary}

\begin{proof}
Suppose that $A=I+u$ for a nilpotent matrix $u.$ Lemma \ref{jordan} implies
that $A^{2}$ is conjugate to a direct sum of Jordan blocks. Without loss of
generality, we assume $A$ is a Jordan block. Then $A^{2}=I+2u+u^{2}.$ By
Lemma \ref{jordan} again, $A^{2}$ is conjugate to a direct sum of Jordan
blocks with $1s$ along the diagonal. The minimal polynomial of $A^{2}$ is $%
(x-1)^{n},$ which shows that there is only one block in the direct sum and
thus $A^{2}$ is conjugate to $A.$
\end{proof}

Recall that a matrix $A\in \mathrm{GL}_{n}(\mathbb{R})$ is called semisimple
if as a complex matrix $A$ is conjugate to a diagonal matrix. A semisimple
matrix $A$ is elliptic (respectively, hyperbolic) if all its (complex)
eigenvalues have modulus $1$ (respectively, are $>0$). The following lemma
is the complete multiplicative Jordan (or Jordan-Chevalley) decomposition
(cf. \cite{hag}, Lemma 7.1, page 430).

\begin{lemma}
\label{mjordan}Each $A\in $ $\mathrm{GL}_{n}(\mathbb{R})$ can be uniquely
written as $A=ehu,$ where $e,h,u\in \mathrm{GL}_{n}(\mathbb{R})$ are
elliptic, hyperbolic and unipotent, respectively, and all three commute.
\end{lemma}

The following result characterizes the continuous length functions on
compact Lie groups.

\begin{lemma}
\label{compactl}Let $G$ be a compact connected Lie group and $l$ a
continuous length function on $G$. Then $l=0.$
\end{lemma}

\begin{proof}
For any element $g\in G,$ there is a maximal torus $T\backepsilon g.$ For
finite order element $h\in T,$ we have $l(h)=0.$ Note that the set of
finite-order elements is dense in $T.$ Since $l$ is continuous, $l$ vanishes
on $T$ and thus $l(g)=0$ for any $g.$
\end{proof}

\bigskip

\begin{theorem}
\label{1}Let $G=\mathrm{SL}_{n}(\mathbb{R})$ $(n\geq 2).$ Let $%
l:G\rightarrow \lbrack 0,+\infty )$ be a length function, which is
continuous on compact subgroups and the subgroup of diagonal matrices with
positive diagonal entries. Then $l$ is uniquely determined by its images on
the subgroup $D$ of diagonal matrices with positive diagonal entries.
\end{theorem}

\begin{proof}
For any $g\in \mathrm{SL}_{n}(\mathbb{R}),$ let $g=ehu$ be the Jordan
decomposition for commuting elements $e,h,u$, where $e$ is elliptic, $h$ is
hyperbolic and $u$ is unipotent (see Lemma \ref{mjordan}) after
multiplications by suitable powers of determinants. Then 
\begin{equation*}
l(g)\leq l(e)+l(h)+l(u).
\end{equation*}%
For any unipotent matrix $u,$ there is an invertible matrix $a$ such that $%
aua^{-1}$ is strictly upper triangular (see \cite{hag}, Theorem 7.2, page
431). Lemma \ref{unip} implies that $u^{2}$ is conjugate to $u$. Therefore, $%
l(u)=0$ by Lemma \ref{bs}. Since $l$ vanishes on a compact Lie group (cf.
Lemma \ref{compactl}), we have that $l(e)=0$ for any elliptic matrix $e.$
Therefore, $l(g)\leq l(h).$ Similarly, 
\begin{equation*}
l(h)=l(e^{-1}gu^{-1})\leq l(g)
\end{equation*}%
which implies $l(g)=l(h).$ Note that a hyperbolic matrix is conjugate to a
real diagonal matrix with positive diagonal entries.
\end{proof}

\bigskip Let $h_{1i}(x)$ ($i=2,\ldots ,n$) be an $n\times n$ diagonal matrix
whose $(1,1)$-entry is $x,$ $(i,i)$-entry is $x^{-1},$ while other diagonal
entries are $1$s and non-diagonal entries are $0$s.

\begin{proof}[Proof of Theorem \protect\ref{th0}]
By Theorem \ref{1}, the length function $l$ is determined by its image on
the subgroup $D$ generated by 
\begin{equation*}
h_{12}(x)=%
\begin{bmatrix}
x & 0 \\ 
0 & x^{-1}%
\end{bmatrix}%
,x\in \mathbb{R}_{>0}.
\end{equation*}
Take $x=e^{t},t\in \mathbb{R}$. We have $l(h_{12}(e^{\frac{k}{l}}))=\frac{|k|%
}{|l|}l(h_{12}(e))$ for any rational number $\frac{k}{l}.$ Since $l$
satisfies condition 2) of the definition and is continuous on $D$, we see
that $l|_{D}$ is determined by the image $l(h_{12}(e))$ (actually, any real
number $t$ is a limit of a rational sequence). Note that the translation
function $\tau $ vanishes on compact subgroups and is continuous on the
subgroup of diagonal subgroups with positive diagonal entries (cf. \cite{bh}%
, Cor. 10.42 and Ex. 10.43, page 320). Therefore, $l$ is {}proportional to $%
\tau .$ Actually, $\tau $ can be determined explicitly by the formula $%
\mathrm{tr}(A)=\pm 2$\textrm{cosh}$\frac{\tau (A)}{2}$ (for nonzero $\tau
(A) $), where $\mathrm{tr}$ is the trace and \textrm{cosh} is the hyperbolic
cosine function (see \cite{bea}, Section 7.34, page 173). This implies that $%
l(A)$ is determined by the spectral radius of $A$ (which could also be seen
clearly by the matrix norm).
\end{proof}

\bigskip

The subgroup $D<\mathrm{SL}_{n}(\mathbb{R})$ of diagonal matrices with
positive diagonal entries is isomorphic to $(\mathbb{R}_{>0})^{n-1}$ and $D$
is generated by the matrices $h_{1i}(x)$ ($i=2,\ldots ,n$) whose $(1,1)$%
-entry is $x,$ $(i,i)$-entry is $x^{-1}.$ Since $h_{1i}(x)$ $(i\neq 1)$ is
conjugate to $h_{12}(x),$ a length function $l:\mathrm{SL}_{n}(\mathbb{R}%
)\rightarrow \lbrack 0,+\infty )$ is completely determined by its image on
the convex hull spanned by $h_{12}(e),h_{13}(e),\ldots ,h_{1n}(e)$ (see
Theorem \ref{lie} for a more general result on Lie groups). Here $e$ is the
Euler's number in the natural exponential function.

\begin{corollary}
\label{cornavas}Let $l:\mathrm{SL}_{2}(\mathbb{R})\rightarrow \lbrack
0,+\infty )$ be a non-trivial length function that is continuous on the
subgroup $SO(2)$ and the diagonal subgroup. Then $l(g)>0$ if and only if $g$
is hyperbolic.
\end{corollary}

\begin{proof}
It is well-known that the elements in $\mathrm{SL}_{2}(\mathbb{R})$ are
classified as elliptic, hyperbolic and parabolic elements. Moreover, the
translation length $\tau $ vanishes on the compact subgroup $SO(2)$ and the
parabolic elements. The corollary follows from Theorem \ref{th0}.
\end{proof}

When the length function $l$ is the asymptotic distortion function $\mathrm{%
dist}_{\infty }$, Corollary \ref{cornavas} is known to Navas \cite{en}
(Proposition 4).

\section{Length functions on algebraic and Lie groups\label{section7}}

For an algebraic group $G,$ let $k[G]$ be the regular ring. For any $g\in G,$
let $\rho _{g}:k[G]\rightarrow k[G]$ be the right translation by $g.$ The
following is the famous Jordan (or Jordan-Chevalley) decomposition.

\begin{lemma}
(\cite{hum1} p.99) Let $G$ be an algebraic group and $g\in G.$ There exists
unique elements $g_{s},g_{u}$ such that $g_{s}g_{u}=g_{u}g_{s},$ and $\rho
_{g_{s}}$ is semisimple, $\rho _{g_{u}}$ is unipotent.
\end{lemma}

\begin{lemma}
\label{6.2}(\cite{luts}) Let $G$ be a reductive connected algebraic group
over an algebraically closed field $k.$ The conjugacy classes of unipotent
elements in $G$ are finite.
\end{lemma}

\begin{lemma}
\label{6.3}(\cite{gar}, Theorem 3.4) If $G$ is a reductive linear algebraic
group defined over a field $k$ and $g\in G(k)$ then the set of conjugacy
classes in $G(k)$ which when base changed to the {}algebraically closed
field $\bar{k}$ are equal to the conjugacy class of $g$ in $G(\bar{k})$ is
in bijection with a subset of $H^{1}(\bar{k}/k,Z(g)(k)),$ the Galois
cohomology set.
\end{lemma}

\begin{definition}
A field $k$ is of type (F) if for any integer $n$ there exist only finitely
many extensions of $k$ of degree $n$ (in a fixed algebraic closure $\bar{k}$
of $k$).
\end{definition}

Examples of fields of type (F) include: the field $\mathbb{R}$ of reals, a
finite field, the field of formal power series over an algebraically closed
field.

\begin{lemma}
\label{6.5}[Borel-Serre \cite{bs}, Theorem 6.2] Let $k$ be a field of type
(F) and let $H$ be a linear algebraic group defined over $k$. The set $H^{1}(%
\bar{k}/k,H(k))$ is finite.
\end{lemma}

\begin{lemma}
\label{lemsemisimple}Let $G(k)$ be a reductive linear algebraic group over a
perfect field of type (F) and $l$ a length function on $G=G(k).$ Then $%
l(g)=l(g_{s}),$ where $g_{s}$ is the semisimple part of $g.$
\end{lemma}

\begin{proof}
By the Jordan decomposition $g=g_{s}g_{u},$ we have $l(g)\leq
l(g_{s})+l(g_{u})$ and $l(g_{s})\leq l(g)+l(g_{u}^{-1}).$ Note that for any
integer $n,$ $g_{u}^{n}$ is also unipotent. By the Lemma \ref{6.2}, Lemma %
\ref{6.3} and Lemma \ref{6.5}, there are only finitely many conjugacy
classes of unipotent elements. This implies that $%
hg_{u}^{n_{1}}h^{-1}=g_{u}^{n_{2}}$ for distinct positive integers $%
n_{1},n_{2},$ and some $h\in G.$ Therefore, we have $%
n_{1}l(g_{u})=n_{2}l(g_{u}),$ which implies that $l(g_{u})=0$ and thus $%
l(g)=l(g_{s}).$
\end{proof}

\bigskip

A Lie group $G$ is semisimple if its maximal connected solvable normal
subgroup is trivial. Let $\mathfrak{g}$ be its Lie algebra and let $\exp :%
\mathfrak{g}\rightarrow G$ denote the exponential map. An element $x\in 
\mathfrak{g}$ is real semi-simple if $Ad(x)$ is diagonalizable over $\mathbb{%
R}$. An element $g\in G$ is called hyperbolic (resp. unipotent) if $g$ is of
the form $g=\exp (x)$ where $x$ is real semi-simple (resp. nilpotent). In
either case, the element $x$ is easily seen to be unique and we write $%
x=\log g$. The following is the Jordan decomposition in Lie groups. An
element $e\in G$ is elliptic if $Ad(e)$ is diagonalizable over $\mathbb{C}$
with eigenvalues $1$.

\begin{lemma}
\label{kostant2.1}(\cite{kos}, Prop. 2.1 and Remark 2.1)

\begin{enumerate}
\item Let $g\in G$ be arbitrary. Then g may be uniquely written 
\begin{equation*}
g=e(g)h(g)u(g)
\end{equation*}%
where $e(g)$ is elliptic, $h(g)$ is hyperbolic and $u(g)$ is unipotent and
where the three elements $e(g),h(g),u(g)$ commute.

\item An element $f\in G$ commutes with $g$ if and only if $f$ commutes with
the three components. Moreover, if $f,g$ commutes, then%
\begin{equation*}
e(fg)=e(f)e(g),h(fg)=h(f)h(g),u(fg)=u(f)u(g).
\end{equation*}
\end{enumerate}
\end{lemma}

\begin{lemma}
\label{eber1.14}(\cite{Ebe}, Prop. 1.14.6, page 63) Let $G$ be a connected
semisimple Lie group whose center is trivial. Then there exists an integer $%
n\geq 2$ and an algebraic group $G^{\ast }<\mathrm{GL}_{n}(\mathbb{C})$
defined over $\mathbb{Q}$ such that $G$ is isomorphic to $G_{\mathbb{R}%
}^{\ast 0}$ (the connected component of $G_{\mathbb{R}}^{\ast }$ containing
the identity) as a Lie group.
\end{lemma}

\bigskip

Let $G=KAN$ be an Iwasawa decomposition. The Weyl group $W$ is the finite
group defined as the quotient of the normalizer of $A$ in $K$ modulo the
centralizer of $A$ in $K.$ For an element $h\in A,$ let $W(h)$ be the set of
all elements in $A$ which are conjugate to $h$ in $G.$

\begin{lemma}
\label{hyperbolicconju}(\cite{kos}, Prop. 2.4) An element $h\in G$ is
hyperbolic if and only if it is conjugate to an element in $A.$ In such a
case, $W(h)$ is a single $W$-orbit in $A.$
\end{lemma}

\begin{theorem}
\label{lie}Let $G$ be a connected semisimple Lie group whose center is
finite with an Iwasawa decomposition $G=KAN$. Let $W$ be the Weyl group.

\begin{enumerate}
\item[(i)] Any length function $l$ on $G$ that is continuous on the maximal
compact subgroup $K$ is determined by its image on $A.$

\item[(ii)] Conversely, any length function $l$ on $A$ that is $W$-invariant
(i.e. $l(w\cdot a)=l(a)$) can be extended to be a length function on $G$
that vanishes on the maximal compact subgroup $K.$
\end{enumerate}
\end{theorem}

\begin{proof}
(i) For any $g\in G,$ the Jordan decomposition gives $g=ehu,$ where $e$ is
elliptic, $h$ is hyperbolic and $u$ is unipotent and where the three
elements $e,h,u$ commute (cf. Lemma \ref{kostant2.1}). We assume that $G$
has the trivial center. By Lemma \ref{eber1.14}, the Lie group $G$ is an
algebraic group. Lemma \ref{lemsemisimple} implies that $l$ vanishes on
unipotent elements and $l(g)=l(eh).$ Since $l$ vanishes on $e$ (cf. Lemma %
\ref{compactl}), we have $l(g)=l(h).$ Therefore, the function $l$ is
determined by its image on $A.$ Let $Z$ be the center of $G.$ Then $G/Z$ is
connected with trivial center. For any $z\in Z,g\in G,$ we have $l(z)=0$ and 
$l(gz)=l(g).$ For any $[g]\in G/Z,$ define $l^{\prime }([g])=l(g).$ The
length function $l$ factors through a length function $l^{\prime }$ on $G/Z$
(By Corollary \ref{cor1.1}, it is enough to show that $l^{\prime }$ is
subadditive for commuting pair $a,b\in G/Z.$ But $l^{\prime
}(ab)=l(h_{a}h_{b}u_{a}u_{b})\leq l(h_{a}u_{a})+l(h_{b}u_{b})=l^{\prime
}(a)+l^{\prime }(b),$ where $h_{a},h_{b},$ $u_{a},u_{b}$ are hyperbolic,
unipotent parts of preimages of $a^{\prime },b^{\prime },$ respectively.)

(ii) Let $l$ be a length function $l$ on $A$ that is $W$-invariant. We first
extend $l$ to the set $H$ of all the conjugates of $A.$ For any $g\in G,a\in
A,$ define $l^{\prime }(gag^{-1})=l(a).$ If 
\begin{equation*}
g_{1}a_{1}g_{1}^{-1}=g_{2}a_{2}g_{2}^{-1}
\end{equation*}
for $g_{1},g_{2}\in G,a_{1},a_{2}\in A,$ then $%
g_{1}^{-1}g_{2}a_{2}g_{2}^{-1}g_{1}=a_{1}.$ By Lemma \ref{hyperbolicconju},
there exists an element $w\in W$ such that $w\cdot a_{2}=a_{1}.$ Therefore,
we have $l(a_{1})=l(a_{2})$ and thus $l^{\prime }$ is well-defined on the
set $H$ of conjugates of elements in $A.$ Such a set $H$ is the set
consisting of hyperbolic elements by Lemma \ref{hyperbolicconju}. We then
extend $l^{\prime }$ on the set of all conjugates of elements in $K.$ For
any $g\in G,k\in K,$ define $l^{\prime }(gkg^{-1})=0.$ If 
\begin{equation*}
g_{1}kg_{1}^{-1}=g_{2}ag_{2}^{-1}
\end{equation*}
for $g_{1},g_{2}\in G,k\in K,a\in A,$ then $%
g_{1}^{-1}g_{2}ag_{2}^{-1}g_{1}=k.$ Then $k$ is both hyperbolic and
elliptic. The only element which is both elliptic and hyperbolic is the
identity element. Therefore, we have $k=a=1$ and 
\begin{equation*}
l^{\prime }(g_{1}kg_{1}^{-1})=l(g_{2}ag_{2}^{-1})=1.
\end{equation*}
This shows that $l^{\prime }$ is well-defined on the set of hyperbolic
elements and elliptic elements. For any unipotent element $u\in G,$ define $%
l^{\prime }(u)=0.$ For any element $g\in G,$ let $g=ehu$ be the Jordan
decomposition. Define $l^{\prime }(g)=l^{\prime }(h).$

We check the function $l^{\prime }$ is a length function on $G.$ The
definition shows that $l^{\prime }$ is conjugate invariant. For any positive
integer $n$ and any $g\in G$ with Jordan decomposition $g=ehu,$ we have $%
g^{n}=e^{n}h^{n}u^{n}$ and thus $l^{\prime }(g^{n})=l^{\prime }(h^{n}).$ But 
$h^{n}$ is hyperbolic and conjugate to an element in $A$ (see Lemma \ref%
{hyperbolicconju}). Therefore, we have $l^{\prime }(h^{n})=|n|l^{\prime }(h)$
and thus $l^{\prime }(g^{n})=|n|l^{\prime }(g).$ If $g_{1}=e_{1}h_{1}u_{1}$
commutes with $g_{2}=e_{2}h_{2}u_{2},$ then 
\begin{equation*}
g_{1}g_{2}=e_{1}e_{2}h_{1}h_{2}u_{1}u_{2}
\end{equation*}%
(cf. Lemma \ref{kostant2.1}) and $l^{\prime }(g_{1}g_{2})=l^{\prime
}(h_{1}h_{2}).$ Since $h_{1},h_{2}$ are commuting hyperbolic elements, they
are {}conjugate simultaneously to elements in $A.$ Therefore, we have 
\begin{equation*}
l^{\prime }(h_{1}h_{2})\leq l^{\prime }(h_{1})+l^{\prime }(h_{2})
\end{equation*}%
and 
\begin{equation*}
l^{\prime }(g_{1}g_{2})\leq l^{\prime }(g_{1})+l^{\prime }(g_{2}).
\end{equation*}
\end{proof}

\bigskip

\begin{remark}
A length function $l$ on $A$ is determined by a group homomorphism $%
f:A\rightarrow \mathbb{B},$ for a real Banach space $(\mathbb{B},\Vert \Vert
),$ satisfying $l(a)=\Vert f(a)\Vert $ (see Lemma \ref{tao}). The previous
theorem implies that a length function $l$ on the Lie group $G$ (that is
continuous on compact subgroup) is uniquely determined by such a group
homomorphism $f:A\rightarrow \mathbb{B}$ such that $\Vert f(a)\Vert =\Vert
f(wa)\Vert $ for any $a\in A$ and $w\in W,$ the Weyl group.
\end{remark}

Let $G$ be a connected semisimple Lie group whose center is finite with an
Iwasawa decomposition $G=KAN$. Let $\exp :\mathfrak{g}\rightarrow G$ be the
exponent map from the Lie algebra $\mathfrak{g}$ with subalgebra $\mathfrak{h%
}$ corresponding to $A$.

\begin{theorem}
Suppose that $l$ is a length function on $G$ that is continuous on $K$ and $%
A.$ Then $l$ is determined by its image on $\exp (v)$ $($unit vector $v\in 
\mathfrak{h})$ in a fixed closed Weyl chamber of $A.$
\end{theorem}

\begin{proof}
Let $Z$ be the center of $G.$ Then $G/Z$ is connected with trivial center.
The length function $l$ factors through a length function on $G/Z$. For any $%
g\in G$ we have $g=ehu,$ where $e$ is elliptic, $h$ is hyperbolic and $u$ is
unipotent and where the three elements $e,h,u$ commute (cf. Lemma \ref%
{kostant2.1}). By Lemma \ref{eber1.14} and Lemma \ref{lemsemisimple}, we
have $l(g)=l(eh).$ Since $l$ vanishes on $e$ (cf. Lemma \ref{compactl}), we
have $l(g)=l(h).$ Any element $h\in A$ is conjugate to an element in a fixed
Weyl chamber $C$ (cf. \cite{hum2}, Theorem 8.20, page 254). For any element $%
\exp (x)\in C,$ with unit vector $x\in \mathfrak{h},$ the one-parameter
subgroup $\exp (\mathbb{R}x)$ lies in $A.$ Since $l$ is continuous on $A,$
the function $l$ is determined by its image on $\exp (\mathbb{Q}x).$ Note
that $l(e^{\frac{m}{n}x})=\frac{1}{n}l(e^{mx})=\frac{m}{n}l(e^{x})$ for any
rational number $\frac{m}{n}.$ The function $l$ is determined by $l(e^{x}),$
for all unit vectors $x$ in the fixed closed Weyl chamber.
\end{proof}

\begin{corollary}
Let $G$ be a connected semisimple Lie group whose center is finite of real
rank $1.$ There is essentially only one length function on $G.$ In order
words, any continuous length function is proportional to the translation
function on the symmetric space $G/K.$
\end{corollary}

\begin{proof}
When the real rank of $G$ is $1,$ a closed Weyl chamber is of dimension $1$.
Therefore, the previous theorem implies that any continuous length function
is determined by its image on a unit vector in a split torus.
\end{proof}

\section{Rigidity of group homomorphisms on arithmetic groups\label{section8}%
}

Let $V$ denote a finite-dimensional vector space over $\mathbb{C}$, endowed
with a $\mathbb{Q}$-structure. Recall that the arithmetic subgroup is
defined as the following (cf. Borel \cite{borel}, page 37).

\begin{definition}
Let $G$ be a $\mathbb{Q}$-subgroup of $\mathrm{GL}(V)$. A subgroup $\Gamma $
of $G_{\mathbb{Q}}$ is said to be arithmetic if there exists a lattice $L$
of $V_{\mathbb{Q}}$ such that $\Gamma $ is commensurable with $G_{L}=\{g\in
G_{\mathbb{Q}}:gL=L\}.$
\end{definition}

Let $\Gamma $ be an arithmetic subgroup of a simple algebraic $\mathbb{Q}$%
-group of $\mathbb{Q}$-rank at least $2.$ Suppose that $G$ is virtually
poly-positive. We will prove Theorem \ref{th0.2}, i.e. any group
homomorphism $f:\Gamma \rightarrow G$ has its image finite.

Recall that a group is quasi-simple, if any non-trivial normal subgroup is
either finite or of finite index. The Margulis-Kazhdan theorem (see \cite%
{z84}, Theorem 8.1.2) implies that an irreducible lattice (and hence) in a
semisimple Lie group of real rank $\geq 2$ is quasi-simple.

\begin{lemma}
\label{lemaheis}Let $\Gamma $ be a finitely generated quasi-simple group
that contains a Heisenberg subgroup, i.e. there are elements torsion-free
elements $a,b,c\in \Gamma $ satisfying $[a,b]=c,[a,c]=[c,b]=1$. Suppose that 
$G$ is virtually poly-positive. Then any group homomorphism $f:\Gamma
\rightarrow G$ has its image finite.
\end{lemma}

\begin{proof}
Suppose that $G$ has a finite-index subgroup $H$ and a subnormal series%
\begin{equation*}
1=H_{n}\vartriangleleft H_{n-1}\vartriangleleft \cdots \vartriangleleft
H_{0}=H
\end{equation*}%
such that every finitely generated subgroup of $H_{i}/H_{i+1}$ has a purely
positive length function. Without loss of generality, we assume that $H$ is
normal. Let $f:\Gamma \rightarrow G$ be a homomorphism. The kernel of the
composite%
\begin{equation*}
f_{0}:\Gamma \overset{f}{\rightarrow }G\rightarrow G/H
\end{equation*}%
is finitely generated. Suppose that the image of the composite%
\begin{equation*}
f_{1}:\ker f_{0}\overset{f}{\rightarrow }H\rightarrow H/H_{1}
\end{equation*}%
has a purely positive length function $l.$ After passing to finite-index
subgroups, we may still suppose that $\ker f_{0}$ contains a Heisenberg
subgroup $\langle a,b,c\rangle .$ By Lemma \ref{nilpfac}, the length
function $l$ vanishes on $f_{1}(c).$ Therefore $f_{1}(c^{k})=1\in H/H_{1}$
for some integer $k>0.$ The normal subgroup $\ker f_{1}$ containing $c^{k}$
is of finite index. Now we have map $\ker f_{1}\rightarrow H_{1}$ induced by 
$f.$ An induction argument shows that $f$ maps some power $c^{d}$ of the
central element of the Heisenberg subgroup into the identity $1\in G.$
Therefore, the image of $f$ is finite.
\end{proof}

\begin{proof}[Proof of Theorem \protect\ref{th0.2}]
It is well-known that $\Gamma $ contains a $\mathbb{Q}$-split simple
subgroup whose root system is the reduced subsystem of the $\mathbb{Q}$-root
system of $\Gamma $ (see \cite{bt}, Theorem 7.2, page 117). Replacing $%
\Gamma $ with this $\mathbb{Q}$-subgroup, we may assume $\Gamma $ is $%
\mathbb{Q}$-split and the root system of $\Gamma $ is reduced. Because $%
\Gamma $ is simple and $\mathbb{Q}$-rank($\Gamma $)$\geq 2$, we know that
the $\mathbb{Q}$-root system of $\Gamma $ is irreducible and has rank at
least two. Therefore, the $\mathbb{Q}$-root system of $\Gamma $ contains an
irreducible subsystem of rank two, that is, a root subsystem of type $%
A_{2},B_{2},G_{2}$ (see \cite{witte}, page 338). For $A_{2},$ choose $%
\{\alpha _{1},\alpha _{3}\}$ as a set of simple roots (see Figure 1). 
\begin{equation*}
\FRAME{itbpF}{2.2667in}{3.25in}{0in}{}{}{qby3yu00.bmp}{\special{language
"Scientific Word";type "GRAPHIC";maintain-aspect-ratio TRUE;display
"USEDEF";valid_file "F";width 2.2667in;height 3.25in;depth
0in;original-width 7.3059in;original-height 10.5135in;cropleft "0";croptop
"1";cropright "1";cropbottom "0";filename 'QBY3YU00.bmp';file-properties
"XNPEU";}}
\end{equation*}%
Then the root element $x_{\alpha _{1}+\alpha _{3}}(rs)=x_{\alpha _{2}}(rs)$
is a commutator $[x_{\alpha _{1}}(r),x_{\alpha _{3}}(s)],$ with $x_{\alpha
_{2}}(rs)$ commutes with $x_{\alpha _{1}}(r),x_{\alpha _{3}}(s).$ For $G_{2}$%
, the long roots form a subsystem of $A_{2}.$ For $B_{2},$ choose $\{\alpha
_{1},\alpha _{4}\}$ as a set of simple roots (see Figure 2). The long root
element $x_{\alpha _{3}}(2rs)$ is a commutator $[x_{\alpha
_{2}}(r),x_{\alpha _{4}}(s)]$ of the two short root elements, and $x_{\alpha
_{3}}(2rs)$ commutes with $x_{\alpha _{2}}(r),x_{\alpha _{4}}(s)$ (cf. \cite%
{hum1}, Proposition of page 211). This shows that the arithmetic subgroup $%
\Gamma $ contains a Heisenberg subgroup. The theorem is then proved by Lemma %
\ref{lemaheis}.
\end{proof}

If we consider special length functions, general results can be proved. When
we consider the stable word lengths, the following is essentially already
known (cf. Polterovich \cite{po}, Corollary 1.1.D and its proof).

\begin{proposition}
\label{polyword}Let $\Gamma $ be an irreducible non-uniform lattice in a
semisimple connected, Lie group without compact factors and with finite
center of real rank $\geq 2.$ Assume that a group $G$ has a virtually
poly-positive stable word length. In other words, the group $G$ has a
finite-index subgroup $H$ and a subnormal series%
\begin{equation*}
1=H_{n}\vartriangleleft H_{n-1}\vartriangleleft \cdots \vartriangleleft
H_{0}=H
\end{equation*}%
such that every finitely generated subgroup of $H_{i}/H_{i+1}$ ($%
i=0,1,\ldots ,n-1$) has a purely positive stable word length. Then any group
homomorphism $f:\Gamma \rightarrow G$ has its image finite.
\end{proposition}

\begin{proof}
Without loss of generality, we assume that $f$ takes its image in $H.$ Since
a lattice is finitely generated, $\Gamma $ has its image in $H_{0}/H_{1}$
finitely generated. When the image has a purely positive word length, any
distorted element in $\Gamma $ must have a trivial image in $H_{0}/H_{1}$
(see Lemma \ref{1.2}). Lubotzky, Mozes and Raghunathan \cite{lmr} prove that
irreducible non-uniform lattices in higher rank Lie groups have non-trivial
distortion elements (They prove the stronger result that there are elements
in the group whose word length has logarithmic growth). Then a finite-index
subgroup $\Gamma _{0}<\Gamma $ will have image in $H_{1}$, since high-rank
irreducible lattices are quasi-simple. An induction argument finishes the
proof.
\end{proof}

\bigskip

When we consider the length given by quasi-{}cocycles, the following is also
essentially already known (cf. Py \cite{py}, Prop. 2.2, following
Burger-Monod \cite{bm1} \cite{bm2}). Recall that a locally compact group has
property TT if any continuous rough action on a Hilbert space has bounded
orbits (see \cite{monod}, page 172). Burger-Monod proves that an irreducible
lattice $\Gamma $ in a high-rank semisimple Lie group has property TT.

\begin{proposition}
\label{polycocycle}Let $\Gamma $ be an irreducible lattice in a semisimple
connected, Lie group without compact factors and with a finite center of
real rank $\geq 2.$. Assume that a group $G$ has a virtually poly-positive
average norm for quasi-cocycles. In other words, the group $G$ has a
finite-index subgroup $H$ and a subnormal series%
\begin{equation*}
1=H_{n}\vartriangleleft H_{n-1}\vartriangleleft \cdots \vartriangleleft
H_{0}=H
\end{equation*}%
such that every finitely generated subgroup of $H_{i}/H_{i+1}$ ($%
i=0,1,\ldots ,n-1$) has a purely positive length given by a quasi-cocycle
with values in Hilbert spaces. Then any group homomorphism $f:\Gamma
\rightarrow G$ has its image finite.
\end{proposition}

\begin{proof}
Note that a group $\Gamma $ has property TT if and only if $H^{1}(\Gamma
;E)=0$ and 
\begin{equation*}
\ker (H_{b}^{2}(\Gamma ;E)\rightarrow H^{2}(\Gamma ;E))=0
\end{equation*}%
for any linear isometric action of $\Gamma $ on a Hilbert space $E.$ Here $%
H_{b}^{2}(\Gamma ;E)$ is the second bounded cohomology group. Suppose that $%
u:\Gamma \rightarrow E$ is a quasi-{}cocycle. There is a bounded map $%
v:\Gamma \rightarrow E$ and a $1$-cocycle $w:\Gamma \rightarrow E$ such that 
\begin{equation*}
u=v+w,
\end{equation*}%
by Proposition 2.1 of Py \cite{py}. Since $\Gamma $ has property T, there
exists $x_{0}\in E$ such that $w(\gamma )=\gamma x_{0}-x_{0}.$ Therefore, we
have%
\begin{eqnarray*}
\frac{\Vert u(\gamma ^{n})\Vert }{n} &=&\frac{\Vert v(\gamma ^{n})+w(\gamma
^{n})\Vert }{n} \\
&=&\frac{\Vert v(\gamma ^{n})+\gamma ^{n}x_{0}-x_{0}\Vert }{n} \\
&\leq &\frac{\Vert v(\gamma ^{n})\Vert +2\Vert x_{0}\Vert }{n}\rightarrow 0.
\end{eqnarray*}

Without loss of generality$,$ we assume that $G=H.$ Suppose that any
finitely generated subgroup of $H/H_{1}$ has a purely positive average norm $%
l$ given by a cocycle. The composite 
\begin{equation*}
\Gamma \overset{f}{\rightarrow }H\rightarrow H/H_{1}
\end{equation*}
has a finite-index kernel $\Gamma _{0}$, since $l$ vanishes on
infinite-order elements of the image. This implies that $f(\Gamma _{0})$
lies in $H_{1}.$ A similar argument proves that $\ker f$ is of finite index
in the general case.
\end{proof}

\section{Rigidity of group homomorphisms on matrix groups\label{section9}}

\subsection{Steinberg groups over finite rings}

Recall that a ring $R$ is right Artinian if any non-empty family of right
ideals contains minimal elements. A ring $R$ is semi-local if $R/\mathrm{rad}%
(R)$ is right Artinian (see Bass' K-theory book \cite{ba} page 79 and page
86), where $\mathrm{rad}(R)$ is the Jacobson radical. Let $n$ be a positive
integer and $R^{n}$ the free $R$-module of rank $n$ with the standard basis.
A vector $(a_{1},\ldots ,a_{n})$ in $R^{n}$ is called \emph{right unimodular}
if there are elements $b_{1},\ldots ,b_{n}\in R$ such that $%
a_{1}b_{1}+\cdots +a_{n}b_{n}=1$. The \emph{stable range condition} $\mathrm{%
sr}_{m}$ says that if $(a_{1},\ldots ,a_{m+1})$ is a right unimodular vector
then there exist elements $b_{1},\ldots ,b_{m}\in R$ such that $%
(a_{1}+a_{m+1}b_{1},\ldots ,a_{m}+a_{m+1}b_{m})$ is right unimodular. It
follows easily that $\mathrm{sr}_{m}\Rightarrow \mathrm{sr}_{n}$ for any $%
n\geq m$. A semi-local ring has the stable range $\mathrm{sr}_{2}$ ( \cite%
{ba}, page 267, the proof of Theorem 9.1). A finite ring $R$ is right
Artinian and thus has $\mathrm{sr}_{2}.$ The stable range 
\begin{equation*}
\mathrm{sr}(R)=\min \{m:R\text{ has }\mathrm{sr}_{m+1}\}.
\end{equation*}%
Thus $\mathrm{sr}(R)=1$ for a finite ring $R$.

We briefly recall the definitions of the elementary subgroups $E_{n}(R)$ of
the general linear group $\mathrm{GL}_{n}(R)$, and the Steinberg groups $%
\mathrm{St}_{n}(R)$. Let $R$ be an associative ring with identity and $n\geq
2$ be an integer. The general linear group $\mathrm{GL}_{n}(R)$ is the group
of all $n\times n$ invertible matrices with entries in $R$. For an element $%
r\in R$ and any integers $i,j$ such that $1\leq i\neq j\leq n,$ denote by $%
e_{ij}(r)$ the elementary $n\times n$ matrix with $1s$ in the diagonal
positions and $r$ in the $(i,j)$-th position and zeros elsewhere. The group $%
E_{n}(R)$ is generated by all such $e_{ij}(r),$\textsl{\ i.e. }%
\begin{equation*}
E_{n}(R)=\langle e_{ij}(r)|1\leq i\neq j\leq n,r\in R\rangle .
\end{equation*}%
Denote by $I_{n}$ the identity matrix and by $[a,b]$ the commutator $%
aba^{-1}b^{-1}.$

The following lemma displays the commutator formulas for $E_{n}(R)$ (cf.
Lemma 9.4 in \cite{mag}).

\begin{lemma}
\label{ecom}Let $R$ be a ring and $r,s\in R.$ Then for distinct integers $%
i,j,k,l$ with $1\leq i,j,k,l\leq n,$ the following hold:

\begin{enumerate}
\item[(1)] $e_{ij}(r+s)=e_{ij}(r)e_{ij}(s);$

\item[(2)] $[e_{ij}(r),e_{jk}(s)]=e_{ik}(rs);$

\item[(3)] $[e_{ij}(r),e_{kl}(s)]=I_{n}.$
\end{enumerate}
\end{lemma}

By Lemma \ref{ecom}, the group $E_{n}(R)$ $(n\geq 3)$ is finitely generated
when the ring $R$ is finitely generated. Moreover, when $n\geq 3,$ the group 
$E_{n}(R)$ is normally generated by any elementary matrix $e_{ij}(1).$

The commutator formulas can be used to define Steinberg groups as follows.
For $n\geq 3,$ the Steinberg group $\mathrm{St}_{n}(R)$ is the group
generated by the symbols $\{x_{ij}(r):1\leq i\neq j\leq n,r\in R\}$ subject
to the following relations:

\begin{enumerate}
\item[(St$1$)] $x_{ij}(r+s)=x_{ij}(r)x_{ij}(s);$

\item[(St$2$)] $[x_{ij}(r),x_{jk}(s)]=x_{ik}(rs)$ for $i\neq k;$

\item[(St$3)$] $[x_{ij}(r),x_{kl}(s)]=1$ for $i\neq l,j\neq k.$
\end{enumerate}

\noindent There is an obvious surjection $\mathrm{St}_{n}(R)\rightarrow
E_{n}(R)$ defined by $x_{ij}(r)\longmapsto e_{ij}(r).$

For any ideal $I\vartriangleleft R,$ let $p:R\rightarrow R/I$ be the
quotient map. Then the map $p$ induces a group homomorphism $p_{\ast }:%
\mathrm{St}_{n}(R)\rightarrow \mathrm{St}_{n}(R/I).$ Denote by $\mathrm{St}%
_{n}(R,I)$ (\textsl{resp., }$E_{n}(R,I)$) the subgroup of $\mathrm{St}%
_{n}(R) $ (\textsl{resp., }$E_{n}(R)$) normally generated by elements of the
form $x_{ij}(r)$ (\textsl{resp.,} $e_{ij}(r)$) for $r\in I$ and $1\leq i\neq
j\leq n.$ In fact, $\mathrm{St}_{n}(R,I)$ is the kernel of $p_{\ast }$ (cf.
Lemma 13.18 in Magurn \cite{mag} and its proof). However, $E_{n}(R,I)$ may
not be the kernel of $E_{n}(R)\rightarrow E_{n}(R/I)$ induced by $p.$

\begin{lemma}
\label{finitesteibg}When $n\geq \mathrm{sr}(R)+2,$ the natural map $\mathrm{%
St}_{n}(R)\rightarrow \mathrm{St}_{n+1}(R)$ is injective. In particular,
when $R$ is finite, the Steinberg group $\mathrm{St}_{n}(R)$ is finite for
any $n\geq 3.$
\end{lemma}

\begin{proof}
Let $W(n,R)$ be the kernel of the natural map $\mathrm{St}_{n}(R)\rightarrow 
\mathrm{St}_{n+1}(R).$ When $n\geq \mathrm{sr}(R)+2,$ the kernel $W(n,R)$ is
trivial (cf. Kolster \cite{kol}, Theorem 3.1 and Cor. 2.10). When $n$ is
sufficient large, the Steinberg group $\mathrm{St}_{n}(R)$ is the universal
central extension of $E_{n}(R)$ (cf. \cite{weib}, Proposition 5.5.1. page
240). Therefore, the kernel $\mathrm{St}_{n}(R)\rightarrow E_{n}(R)$ is the
second homology group $H_{2}(E_{n}(R);\mathbb{Z}).$ When $R$ is finite, both 
$E_{n}(R)$ and $H_{2}(E_{n}(R);\mathbb{Z})$ are finite. Therefore, the group 
$\mathrm{St}_{n}(R)$ is finite for any $n\geq 3.$
\end{proof}

\subsection{Rigidity of group homomorphisms on matrix groups}

\begin{theorem}
\label{th1'}Suppose that $G$ is a group satisfying that

1) $G$ has a purely positive length function, i.e. there is a length
function $l:G\rightarrow \lbrack 0,\infty )$ such that $l(g)>0$ for any
infinite-order element $g;$ and

2) any torsion abelian subgroup of $G$ is finitely generated.

Let $R$ be an associative ring with identity and $\mathrm{St}_{n}(R)$ the
Steinberg group. Suppose that $S<\mathrm{St}_{n}(R)$ is a finite-index
subgroup. Then any group homomorphism $f:\mathrm{St}_{n}(R)\rightarrow G$
has its image finite when $n\geq 3$.
\end{theorem}

\begin{proof}
Since any ring $R$ is a quotient of a free (non-commutative) ring $\mathbb{Z}%
\langle X\rangle $ for some set $X$ and $\mathrm{St}_{n}(R)$ is functorial
with respect to the ring $R,$ we assume without loss of generality that $R=$ 
$\mathbb{Z}\langle X\rangle $. We prove the case $S=\mathrm{St}_{n}(R)\ $%
first. Let $x_{ij}=\langle x_{ij}(r):r\in R\rangle ,$ which is isomorphic to
the abelian group $R.$ Note that 
\begin{equation*}
\lbrack x_{12}(1),x_{23}(1)]=x_{13}(1)
\end{equation*}%
and $x_{13}(1)$ commutes with $x_{12}(1)$ and $x_{23}(1).$ Lemma \ref%
{nilpfac} implies any length function vanishes on $x_{13}(1).$ By Lemma \ref%
{lemq}, the length $l(f(x_{13}(1)))=0.$ Note that $x_{ij}(r)$ is conjugate
to $x_{13}(r)$ for any $r\in R$ and $i,j$ satisfying $1\leq i\neq j\leq n.$
Since $l$ is purely positive, we get that $f(x_{12}(1))$ is of finite order.
Let $I=\ker f|_{x_{12}}.$ Then $I\neq \varnothing ,$ as $f(x_{12}(1))$ is of
finite order. For any $x\in I,$ and $y\in R,$ we have 
\begin{equation*}
x_{12}(xy)=[x_{13}(x),x_{32}(y)].
\end{equation*}%
Therefore, 
\begin{equation*}
f(x_{12}(xy))=[f(x_{13}(x),f(x_{32}(y)))]=1
\end{equation*}%
and thus $xy\in I.$ Similarly, we have $%
f(x_{12}(yx))=f([x_{13}(y),x_{32}(x)])=1.$ This proves that $I$ is a
(two-sided) ideal. Note that $f(x_{12})=R/I$ is a torsion abelian group. By
the assumption 2), the quotient ring $R/I$ is finite. Let $\mathrm{St}%
_{n}(R,I)$ be the normal subgroup of $\mathrm{St}_{n}(R)$ generated by $%
x_{ij}(r),r\in I.$ There is a short exact sequence%
\begin{equation*}
\begin{array}{ccc}
1\rightarrow \mathrm{St}_{n}(R,I)\rightarrow & \mathrm{St}_{n}(R)\rightarrow
& \mathrm{St}_{n}(R/I)\rightarrow 1.%
\end{array}%
\end{equation*}%
Since $R/I$ is finite, we know that $\mathrm{St}_{n}(R/I)$ is finite by
Lemma \ref{finitesteibg}. This proves that $\func{Im}f$ is finite since $f$
factors through $\mathrm{St}_{n}(R/I)$. For general finite-index subgroup $%
S, $ we assume $S$ is normal in $\mathrm{St}_{n}(R)$ after passing to a
finite-index subgroup of $S.$ A similar proof shows that $S$ contains $%
\mathrm{St}_{n}(R,I)$ for some ideal $I$ with the quotient ring $R/I$
finite. Therefore, the image $\func{Im}f$ is finite.
\end{proof}

\bigskip

\begin{theorem}
\label{th2}Suppose that $G$ is a group having a purely positive length
function $l$. Let $R$ be an associative ring of characteristic zero such
that any nonzero ideal is of a finite index (eg. the ring of algebraic
integers in a number field). Suppose that $S<\mathrm{St}_{n}(R)$ is a
finite-index subgroup of the Steinberg group. Then any group homomorphism $%
f:S\rightarrow G$ has its image finite when $n\geq 3$.
\end{theorem}

\begin{proof}
The proof is similar to that of Theorem \ref{th1'}. Let $I=\ker f|_{x_{12}},$
where $x_{12}=S\cap \langle x_{12}(r):r\in R\rangle .$ Since $R$ is of
characteristic zero and the length $l(f(x_{12}(k)))=0$ for some integer $k,$
we have $f(x_{12}(k))$ is of finite order. Therefore, $f(x_{12}(k^{\prime
}))=1$ for some integer $k^{\prime },$ which proves that the ideal $I$ is
nonzero. Since $I$ is of finite index in $R,$ we get that $\mathrm{St}%
_{n}(R,I)$ is of finite index in $S.$ This finishes the proof.
\end{proof}

Since the natural map $\mathrm{St}_{n}(R)\rightarrow E_{n}(R)$ is
surjective, any group homomorphism $f:E_{n}(R)\rightarrow G$ can be lifted
to be a group homomorphism $\mathrm{St}_{n}(R)\rightarrow G.$ Moreover, a
finite-index subgroup $E$ of $E_{n}(R)$ is lifted to be a finite-index
subgroup $S$ of $\mathrm{St}_{n}(R).$ Theorem \ref{th1} and Theorem \ref%
{th2'} follow from Theorem \ref{th1'} and Theorem \ref{th2}, by inductive
{}arguments on the subnormal series as those of the proofs of Theorem \ref%
{th0.2}.

\bigskip

\begin{proof}[Proof of Corollary \protect\ref{mcor} and Corollary \protect
\ref{mor'}]
For Corollary \ref{mcor}, it is enough to check the two conditions for $G$
in Theorem \ref{th1}. Lemma \ref{pol} proves that $G$ has a purely positive
length function. When $G$ is a $\mathrm{CAT}(0)$ group, (i.e. $G$ acts
properly and cocompactly on a $\mathrm{CAT}(0)$ space), then any solvable
subgroup of $G$ is finitely generated (and actually virtually abelian, see
the Solvable Subgroup Theorem of \cite{bh}, Theorem 7.8, page 249). When $G$
is hyperbolic, it's well-known that $G$ contains finitely many conjugacy
classes of finite subgroups and thus a torsion abelian subgroup is finite
(see \cite{bh}, Theorem 3.2, page 459). Birman-Lubotzky-McCarthy \cite{blm}
proves that any abelian subgroup of the mapping class groups for orientable
surfaces is finitely generated. Bestvina-Handel \cite{besth} proves that
every solvable subgroup of $\mathrm{Out}(F_{k})$ has a finite index subgroup
that is finitely generated and free abelian. When $G$ is the diffeomorphism
group $\mathrm{Diff}(\Sigma ,\omega ),$ there is a subnormal series (see the
proof of Lemma \ref{area-pres})%
\begin{equation*}
1\vartriangleleft \mathrm{Ham}(\Sigma ,\omega )\vartriangleleft \mathrm{Diff}%
_{0}(\Sigma ,\omega )\vartriangleleft \mathrm{Diff}(\Sigma ,\omega ),
\end{equation*}%
with subquotients in $\mathrm{Ham}(\Sigma ,\omega ),$ $H_{1}(\Sigma ,\mathbb{%
R})$ and the mapping class group $\mathrm{MCG}(\Sigma ).$ Any abelian
subgroup of a finitely generated subgroup of these groups is finitely
generated.

Corollary \ref{mor'} follows from Theorem \ref{th2} and Lemma \ref{pol}.
\end{proof}

\begin{remark}
An infinite torsion abelian group may act properly on a simplicial tree (see 
\cite{bh}, Example 7.11, page 250). Therefore, condition 2) in Theorem \ref%
{th1} does not hold for every group $G$ acting properly on a $\mathrm{CAT}%
(0) $ (or a Gromov hyperbolic) space. We don't know whether condition 2) can
be dropped.
\end{remark}

\bigskip

\section{Length functions on Cremona groups\label{section10}}

Let $k$ be a field and $k(x_{1},x_{2},\ldots ,x_{n})$ be the field of
rational functions in $n$ indeterminates over $k.$ It is well-known that the
Cremona group $\mathrm{Cr}_{n}(k)$ is isomorphic to the automorphism group $%
\mathrm{Aut}_{k}(k(x_{1},x_{2},\ldots ,x_{n}))$ of the field $%
k(x_{1},x_{2},\ldots ,x_{n})$.

\begin{lemma}
\label{cre}Let $f:k(x_{1},x_{2},\ldots ,x_{n})\rightarrow
k(x_{1},x_{2},\ldots ,x_{n})$ be given by $f(x_{1})=\alpha
x_{1},f(x_{i})=x_{i}$ for some $0\neq \alpha \in k$ and any $i=2,\cdots ,n.$
Then $f$ lies in the center of a Heisenberg subgroup. In other words, there
exists $g,h\in \mathrm{Cr}_{n}(k)$ such that $[g,h]=ghg^{-1}h^{-1}=f,[g,f]=1$
and $[h,f]=1.$
\end{lemma}

\begin{proof}
Let $g,h:k(x_{1},x_{2},\ldots ,x_{n})\rightarrow k(x_{1},x_{2},\ldots
,x_{n}) $ be given by 
\begin{equation*}
g(x_{1})=x_{1}x_{2},g(x_{i})=x_{i}(i=2,\ldots ,n)
\end{equation*}%
and 
\begin{equation*}
h(x_{1})=x_{1},h(x_{2})=\alpha ^{-1}x_{2},h(x_{j})=x_{j}(j=3,\ldots ,n).
\end{equation*}%
It can be directly checked that $[g,h]=f,[g,f]=1$ and $[h,f]=1.$
\end{proof}

\begin{lemma}
\label{dia}Let $l:\mathrm{Bir}(\mathbb{P}_{k}^{n})\rightarrow \lbrack
0,\infty )$ be a length function $(n\geq 2)$. Then $l$ vanishes on diagonal
elements and unipotent elements of $\mathrm{Aut}(\mathbb{P}_{k}^{n})=\mathrm{%
PGL}_{n+1}(k).$
\end{lemma}

\begin{proof}
Let $g=\mathrm{diag}(a_{0},a_{1},\ldots ,a_{n})\in \mathrm{PGL}_{n+1}(k)$ be
a diagonal element. Note that $l$ is subadditive on the diagonal subgroups.
In order to prove $l(g)=0,$ it is enough to prove that $l(\mathrm{diag}%
(1,\ldots ,1,a_{i},1,\ldots ,1))=0,$ where $\mathrm{diag}(1,\ldots
,1,a_{i},1,\ldots ,1)$ is the diagonal matrix with $a_{i}$ in the $(i,i)$-th
position and all other diagonal entries are $1.$ But $\mathrm{diag}(1,\ldots
,1,a_{i},1,\ldots ,1)$ is conjugate to $\mathrm{diag}(1,\alpha ,1,\ldots ,1)$
for $\alpha =a_{i}.$ Lemma \ref{cre} implies that $\mathrm{diag}(1,\alpha
,1,\ldots ,1)$ lies in the center of a Heisenberg group. Therefore, $l(%
\mathrm{diag}(1,\alpha ,1,\ldots ,1))=0$ by Lemma \ref{nilpfac}. This proves 
$l(g)=0.$ The vanishing of $l$ on unipotent elements follows from Corollary %
\ref{unip} when the characteristic of $k$ is not $2$. When the
characteristic of $k$ is $2,$ any unipotent element $A=I+u$ (where $u$ is
nilpotent) is of finite order. This means $l(A)=0$.
\end{proof}

\bigskip

\begin{proof}[Proof of Theorem \protect\ref{th3}]
When $k$ is algebraically closed, the Jordan normal form implies that any
element $g\in \mathrm{PGL}_{n}(k)$ is conjugate to the form $sn$ with $s$
diagonal and $n$ the strictly upper triangular matrix. Moreover, $sn=ns.$
Therefore, $l(f)\leq l(s)+l(n).$ By Lemma \ref{dia}, $l(s)=l(n)=0$ and thus $%
l(g)=0$.
\end{proof}

\bigskip

\begin{proof}[Proof of Corollary \protect\ref{fincor}]
Let $f:\mathrm{Bir}(\mathbb{P}_{k}^{2})\rightarrow G$ be a group
homomorphism. Suppose that $G$ has a purely positive length function $l.$ By
Theorem \ref{th3}, the purely positive length function $l$ on $G$ will
vanish on $f(\mathrm{PGL}_{3}(k)).$ Since $k$ is infinite and $\mathrm{PGL}%
_{3}(k)$ is a simple group, we get that $\mathrm{PGL}_{3}(k)$ lies in the $%
\ker f.$ By the Noether-Castelnuovo Theorem, $\mathrm{Bir}(\mathbb{P}%
_{k}^{2})$ is generated by $\mathrm{PGL}_{3}(k)$ and an involution.
Moreover, the $\mathrm{Bir}(\mathbb{P}_{k}^{2})$ is normally generated by $%
\mathrm{PGL}_{3}(k).$ Therefore, the group homomorphism $f$ is trivial. The
general case is proved by an inductive argument on the subnormal series of a
finite-index subgroup of $G.$
\end{proof}

\begin{lemma}
Let $\mathrm{Bir}(\mathbb{P}_{\mathbb{R}}^{n})$ $(n\geq 2)$ be the real
Cremona group. Any length function $l:\mathrm{Bir}(\mathbb{P}%
_{k}^{n})\rightarrow \lbrack 0,\infty ),$ which is continuous on $\mathrm{PSO%
}(n+1)<\mathrm{Aut}(\mathbb{P}_{\mathbb{R}}^{n}),$ vanishes on $\mathrm{PGL}%
_{n+1}(\mathbb{R}).$
\end{lemma}

\begin{proof}
By Lemma \ref{dia}, the length function $l$ vanishes on diagonal matrices of 
$\mathrm{PGL}_{n+1}(\mathbb{R}).$ Theorem \ref{th0.1} implies that $l$
vanishes on the whole group $\mathrm{PGL}_{n+1}(\mathbb{R}).$
\end{proof}

\bigskip

\noindent \textbf{Acknowledgements}

The author wants to thank many people for helpful discussions, including
Wenyuan Yang on a discussion of hyperbolic groups, C. Weibel on a discussion
on Steinberg groups of finite rings, Feng Su on a discussion of Lie groups,
Ying Zhang on a discussion of translation lengths of hyperbolic spaces,
Enhui Shi for a discussion on smooth measure-theoretic entropy. The author
is also grateful to Yongle Jiang and Shripad Garge for comments on a
previous version of this article. This work is supported by NSFC (No.
11971389).

\bigskip

NYU Shanghai, 1555 Century Avenue, Shanghai, 200122, China.

NYU-ECNU Institute of Mathematical Sciences at NYU Shanghai, 3663 Zhongshan
Road North, Shanghai, 200062, China

E-mail: sy55@nyu.edu

\end{document}